\documentclass[a4paper,10pt]{amsart}

\usepackage{macros}

\begin{document}
\title[Metric completions from hereditary rings]{Metric completions of triangulated categories from hereditary rings}

\author{\textsc{Cyril Matou\v sek}}
\address{Aarhus University,
Department of Mathematics,
Ny Munkegade 118,
8000 Aar\-hus~C,
Denmark}
\email{cyril.matousek@math.au.dk}

\thanks{This work was supported by a research grant (VIL42076) from VILLUM FONDEN}
\subjclass[2020]{13D09, 16E35, 16G10, 18G80}
\date{\today}

\begin{abstract}
The focus of this article is on metric completions of triangulated categories arising in the representation theory of hereditary finite dimensional algebras and commutative rings.
We explicitly describe all completions of bounded derived categories with respect to additive good metrics for two classes of rings - hereditary commutative noetherian rings and hereditary algebras of tame representation type over an algebraically closed field.
To that end, we develop and study the lattice theory of metrics on triangulated categories. Moreover, we establish a link between metric completions of bounded derived categories of a ring and the ring’s universal localisations.
\end{abstract}

\maketitle
\vspace{4ex}

{
\hypersetup{linkcolor=black}
\tableofcontents
}

\section{Introduction}
\label{sec:introduction}

Neeman's method of constructing a metric completion of a triangulated category, inspired by the Cauchy completion of a metric space, is one of the few (\cite{Balmer11}, \cite{BalmerStevenson21}, \cite{Beligiannis00}, \cite{Keller05}), rather scarce tools to build a new triangulated category from an existing old one. Developed in~\cite{Neeman18} as a particular case of Lawvere's norm on a category (\cite{Lawvere73}), the data of a metric on triangulated category postulate how far individual morphisms of the~category are from being isomorphisms.

Under the interpretation of a morphism close to an isomorphism being short, we can define Cauchy sequences of objects and morphisms as \(\N\)-directed diagrams with morphisms eventually getting shorter and shorter.
A completion of the triangulated category then consists of formal directed colimits of all Cauchy sequences, but restricted to only those objects which perceive short morphisms as actual isomorphisms - we call those special objects ``compactly supported''.
The crucial point of \cite{Neeman18} is that the metric completion satisfyies all the axioms of triangulated categories (\cite{Puppe62},\cite{Verdier96}) with the triangulated structure induced by the~original triangulated structure of the initial category.

Since its introduction in 2018, the technique of metric completions for triangulated categories has found broad range of applications in homological algebra, algebraic geometry, and related areas. In \cite{Neeman18} (and the updated version \cite{Neeman25A}), Neeman generalises Rickard's derived Morita theory results \cite{Rickard89} to noetherian, weakly approximable triangulated categories by expressing two triangulated categories as mutual metric completions of each other.
He then encapsulates this phenomenon by defining excellent metrics on triangulated categories in \cite{Neeman25B} with the property that the process of taking the completion and then the opposite category is an involution. 
In \cite{CummingsGratz24}, Cummings and Gratz calculate metric completions of Igusa-Todorov (\cite{IgusaTodorov15}) discrete cluster categories as subcategories of their Paquette-Yıldırım (\cite{PaquetteYildirim21}) combinatorial completion.
Sun and Zhang study recollements of completions of triangulated categories in \cite{SunZhang21}. 
Biswas, Chen, Manali Rahul, Parker, and Zheng formulate a necessary condition for an existence of a bounded $t$-structure on a~compactly generated triangulated using specific metric completions in \cite{BiswasChenRahulParkerZheng24}.
Manali Rahul (\cite{Rahul25A},\cite{Rahul25B}) also uses metric techniques to obtain semi-orthogonal decompositions of triangulated categories and Brown representability flavoured results in algebraic geometry.   

Despite the outstanding theoretical strength of the theorem establishing a triangulated structure on a metric completion, the completions are rather difficult to compute in practice for a general triangulated category. For example, the calculations of completions in \cite{CummingsGratz24} rely on the combinatorial nature of triangulated categories specific to the~setting of discrete cluster categories. Many other results involving metrics on triangulated categories express the respective completion as a triangulated subcategory of a bigger, ambient category called good extension (see~\ref{sssec:GoodExtensions}), whose existence, however, cannot be guaranteed for an arbitrary triangulated category. What is more, while the~definition of a metric is relatively generous and easy to satisfy, the wast majority of completion-related results utilise only very specific choices of metrics determined by either aisles, or coaisles of $t$\=/structures.

This lack of diverse examples of metrics and the corresponding completions serves as the main motivation behind this article.
We present a systematic study of metrics and metric completions for the bounded derived category \(\derived^b(\modf\dashmodule R)\) for a large class of hereditary rings $R$. We do not limit ourselves to metrics arising from $t$-structures only, but we explicitly describe the completions with respect to all possible additive good metrics on \(\derived^b(\modf\dashmodule R)\).

Technically, the data of a metric on a triangulated category \(\mathcal{S}\) consist of a family \(\mathcal{M}=\{B_n\}_{n\in\N}\) of subcategories of \(\mathcal{S}\) corresponding to cones of short morphisms. We use this specific data to describe the corresponding completions. The~key observation is that for an additive good metric \(\mathcal{M}=\{B_n\}_{n\in\N}\) on \(\derived^b(\modf\dashmodule R)\) the~intersection \(\mathcal{B}:=\bigcap_{n\in\N}B_n\) is a thick subcategory of \(\derived^b(\modf\dashmodule R)\). Such a thick subcategory corresponds to a universal localisation \(R\rightarrow R_{\mathcal{B}}\) uniquely determined by the~property \(\mathcal{B}\otimes^{\mathbf{L}}_{R}R_{\mathcal{B}}=0\) in \(\derived(R)\) (see \cite{KrauseStovicek10}), and there is an embedding of derived categories \(\derived(R_\mathcal{B})\xhookrightarrow{}\derived(R)\). With the existence of such an embedding kept in mind, we prove that the completion \(\mathfrak{S}_{\mathcal{M}}\big(\derived^b(\modf\dashmodule R)\big)\) of \(\derived^b(\modf\dashmodule R)\) with respect to \(\mathcal{M}\) is equivalent either to \(\derived^b(\modf\dashmodule R_{\mathcal{B}})\), or a thick subcategory of \(\mathcal{B}^{\perp}=\Ker\Hom_{\derived^b(\modf\dashmodule R)}(\mathcal{B},\blank)\subseteq\derived^b(\modf\dashmodule R)\) depending on whether the lattice-theoretic condition \(\mathcal{M}=\mathcal{M}_{\infty}\vee\mathcal{B}\) that the metric \(\mathcal{M}\) ``converges locally uniformly'' towards \(\mathcal{B}\) in terms of cohomology (see Section~\ref{sec:lattices}) is satisfied.

\begin{theorem}
\label{T:IntroductionMainTheorem}
Let $R$ be a connected hereditary ring and \(\mathcal{M}=\{B_n\}_{n\in\N}\) an additive good metric on \(\derived^b(\modf\dashmodule R)\).

\begin{enumerate}[label=\arabic*)] 
    \item (Theorem~\ref{T:DedekindCompletions}) Suppose that $R$ is a commutative noetherian ring. Then:
    \begin{enumerate}[label=\roman*)] 
        \item If \(\mathcal{B}\) is countably generated as a thick subcategory and \(\mathcal{M}=\mathcal{M}_{\infty}\vee\mathcal{B}\), then \(\mathfrak{S}_{\mathcal{M}}\big(\derived^b(\modf\dashmodule R)\big)\simeq\derived^b(\modf\dashmodule R_{\mathcal{B}})\).
        \item Otherwise, we have \(\mathfrak{S}_{\mathcal{M}}(\mathcal{S})\simeq\mathcal{B}^{\perp}\cap\,\mathfrak{T}\) where \(\mathfrak{T}\subseteq\derived^b(\modf\dashmodule R)\) is the~thick subcategory generated by all finitely generated torsion modules.
    \end{enumerate}
    \item (Theorem~\ref{T:TameCompletions}) Suppose that $R$ is a finite dimensional algebra of tame representation type over an algebraically closed field. Then:
    \begin{enumerate}[label=\roman*)]
        \item If \(\mathcal{B}\) is countably generated as a thick subcategory and \(\mathcal{M}=\mathcal{M}_{\infty}\vee\mathcal{B}\), then \(\mathfrak{S}_{\mathcal{M}}\big(\derived^b(\modf\dashmodule R)\big)\simeq\derived^b(\modf\dashmodule R_{\mathcal{B}})\).
        \item Otherwise, we have \(\mathfrak{S}_{\mathcal{M}}(\mathcal{S})\simeq\mathcal{B}^{\perp}\cap\,\mathfrak{R}\) where \(\mathfrak{R}\subseteq\derived^b(\modf\dashmodule R)\) is the~thick subcategory generated by all 
        regular modules.
    \end{enumerate}
\end{enumerate}
\end{theorem}

This result complements author's earlier work \cite[Section~3]{Matousek26A} on completions for representation-finite algebras. For every finite dimensional algebra $A$ of finite representation finite, all the completions of \(\derived^b(\modf\dashmodule A)\) are just the thick subcategories of \(\derived^b(\modf\dashmodule A)\). Our current result tell us that once we abandon the~realm of representation-finite algebras and consider a tame algebra $R$ instead, we can actually encounter new elements in the~completion \(\mathfrak{S}'_{\mathcal{M}}\big(\derived^b(\modf\dashmodule R)\big)\), such as the~universal localisation \(R_\mathcal{B}\), which do not lie in the initial category \(\derived^b(\modf\dashmodule R)\).

While the fact that a metric completion of a triangulated category \(\mathcal{S}\) can, in fact, shrink to a subcategory of \(\mathcal{S}\) can seem counter-intuitive at the first glance, the~reason behind such a phenomenon resides in the two-step process of creating the~completion. As mentioned before, unlike the usual metric space completion, or e.g.\ the sequential completion of triangulated category (\cite{Krause20}), we are not only adding directed colimits of Cauchy sequences to our category during the construction, but we are also removing objects which are not compactly supported with respect to the~given metric. We discuss how compactly supported elements look like for bounded derived categories of rings in Section~\ref{sec:smash}, where we present their description using universal localisations for hereditary rings together with partial results for general commutative noetherian rings.

Section~\ref{sec:tools} consists of several auxiliary results about completions for bounded derived categories for finitely generated algebras over commutative noetherian rings, which are both useful for proving Theorem~\ref{T:IntroductionMainTheorem} and of an independent interest. For instance, for a hereditary ring, Proposition~\ref{P:CompletionIsIdempotentComplete} gives us that any completion with respect to an additive good metric is an idempotent complete triangulated category, while Proposition~\ref{P:CountablyGeneratedCompletion} fully describes completions with respect to constant metrics determined by countably generated thick subcategories.

The statement of Theorem~\ref{T:IntroductionMainTheorem} involves a lattice-theory condition on a metric on a triangulated category. The result that metrics form a lattice is new, and is proven in Section~\ref{sec:lattices}, together with further applications and examples linked to the lattices of metrics. In Proposition~\ref{P:DecompostionMetricsCompletion}, we give a precise formulation of the condition of a cohomological locally uniform convergence of a~metric towards a~specific thick subcategory, and show its consequences for the resulting metric completions.

Finally, Section~\ref{sec:CommutativeNoetherian} and Section~\ref{sec:TameAlgebras} are dedicated to proving the main result 
Theorem~\ref{T:IntroductionMainTheorem} in our two primary settings - hereditary commutative noetherian rings and hereditary tame algebras. We analyse additive good metrics on bounded derived categories using the ring's prime spectrum in the commutative case and the~structure of regular modules in the tame algebra case.

\subsubsection*{Acknowledgements}
I wish to express my gratitude to Sira Gratz for her supervision and for her
valuable comments regarding this paper. A special thanks goes to Jan Šťovíček for introducing me to the topic of universal localisations. I would also like to thank Isaac Bird, Charley Cummings, David Nkansah, Greg Stevenson, and Alexandra Zvonareva for useful discussions. This work was supported by a research grant (VIL42076)
from VILLUM FONDEN.

\section{Preliminaries}
\label{sec:preliminaries}

\subsection{Notation and conventions}
If $A$ is an object of a category \(\mathcal{A}\), we will only write \(A\in\mathcal{A}\) instead of \(A\in\Ob(\mathcal{A})\). 
Similarly, if \(\mathcal{A}\) and \(\mathcal{B}\) are categories, we will employ the shortcut \(\mathcal{A}\subseteq\mathcal{B}\) to express that \(\mathcal{A}\) is a subcategory of \(\mathcal{B}\).

For a class of objects \(\mathcal{C}\) of an additive category \(\mathcal{A}\),
the~full subcategory of \(\mathcal{A}\) of objects isomorphic
to~finite coproducts of objects from \(\mathcal{C}\) will be denoted as \(\coprodf(\mathcal{C})\). 

We will always implicitly denote the shift functor of any triangulated category as~\(\Sigma\).
Let \(\mathcal{S}\) be a triangulated category.
By~\(\mathcal{S}^c\) we mean the triangulated subcategory consisting of compact elements of \(\mathcal{S}\).
Given a class of objects \(\mathcal{C}\subseteq\mathcal{S}\), the~notation \(\langle\mathcal{C} \rangle\) stands for the triangulated subcategory of \(\mathcal{S}\) generated by \(\mathcal{C}\), while \(\thick(\mathcal{C})\) stands for the thick subcategory of \(\mathcal{S}\) generated by \(\mathcal{C}\).

All rings are assumed to be unital and associative.
A \(\kk\)-algebra over a commutative ring \(\kk\) is a ring \(R\) with a \(\kk\)-module structure determined by a ring homomorphism \(\kk\rightarrow R\) with the image lying in the centre of $R$.

\subsection{Metrics and completions}

Here we recall the most important important facts and definitions about completions of triangulated categories via metric techniques introduced by Neeman in \cite{Neeman25A} and \cite{Neeman20}. We also make use of additional terminology and notational conventions from \cite{CummingsGratz24}. 
In this subsection, we fix a~triangulated category \(\mathcal{S}\).

\subsubsection{Sequences}

\begin{definition}[{\cite[Definition~3.2, Subsection~3.2]{CummingsGratz24}}]
 An \(\N\)-directed diagram of the form \(E_1\xrightarrow{e_1} E_2 \xrightarrow{e_2} \cdots\) in \(\mathcal{S}\) is called \textit{a sequence} and denoted \(\mathbf{E}=(E_n,e_n)_{n\in\N}\). For \(n\leq m\in\N\), the map \(e_{n,m}\) denotes the composition \(e_m\circ\cdots\circ e_{n+1}\circ e_n\).

Let \(\mathbf{I}=(i_m)_{m\in\N}=i_1<i_2<\cdots\) be a strictly increasing sequence of natural numbers. Then we define \textit{a subsequence} \(\mathbf{E}_{\mathbf{I}}=\left(E_{i_n},e_{i_n,i_{n+1}}\right)_{n\in\N}\) of \(\mathbf{E}\).

Assume that for every \(n\in\N\) we have direct summand \(F_n\) of \(E_n\) together with the canonical inclusion \(\nu_n:F_n\rightarrow E_n\) and the canonical projection \(\pi_n:E_n\rightarrow F_n\). 
Then the sequence \(\mathbf{F}=(F_n,\pi_{n+1}\circ e_n\circ\nu_n)\) is called \textit{a component} of \(\mathbf{E}\).
\end{definition}

\begin{notation}
We denote the category of additive contravariant functors \(\mathcal{S}^{\opposite}\rightarrow\Ab\) as \(\Mod\dashmodule\mathcal{T}\). 
Furthermore, the notation for Yoneda's embedding \(X\mapsto\Hom_{\mathcal{S}}(\blank,X)\) is \(\Yoneda:\mathcal{S}\rightarrow\Mod\dashmodule\mathcal{S}\).
Here \(\Ab\) stands for the category of abelian groups.
\end{notation}

Following \cite[Definition~3.4]{CummingsGratz24}, to a sequence \(\mathbf{E}=(E_n,e_n)_{n\in\N}\) in \(\mathcal{S}\) we associate a~functor \(\moco\mathbf{E}:=\varinjlim \Yoneda(\mathbf{E})\in\Mod\dashmodule\mathcal{T}\) called \textit{the module colimit} of \(\mathbf{E}\), which is the directed colimit of the image of \(\mathbf{E}\) in \(\Mod\dashmodule\mathcal{S}\) under Yoneda's embedding.

\begin{definition}[{\cite[Definition~3.11, Definition~3.16]{CummingsGratz24}}]
Let \(\mathcal{W}\) be a replete subcategory of \(\mathcal{S}\).
Let \(\mathbf{E}=(E_n,e_n)_{n\in\N}\) be a sequence in \(\mathcal{T}\).
We say that \(\mathbf{E}\) \textit{stabilises at} \(\mathcal{W}\) \textit{from} \(N\in\N\) \textit{onwards} if for all \(N\leq n\leq m\) we have \(\cone e_{n,m}\in\mathcal{W}\).
We then simply say that \(\mathbf{E}\) \textit{stabilises at} \(\mathcal{W}\) if it stabilises at \(\mathcal{W}\) from some \(N\in\N\) onwards.

A functor \(F\in\Mod\dashmodule\mathcal{S}\) is \textit{compactly supported} at \(\mathcal{W}\) if \(F(\mathcal{W})=0\). 
In this spirit,
the sequence \(\mathbf{E
}\) is \textit{compactly supported} at \(\mathcal{W}\) if the functor \(\moco\mathbf{E}\) is such.

We say that \(\mathbf{E}\) is \(\mathcal{W}\)\textit{-trivial} if for every \(n\in\N\) the entry \(E_n\) does not have any non-zero direct summands which lie in \(\mathcal{W}\). 
\end{definition}

In particular, an object \(X\in\mathcal{S}\) is \textit{compactly supported} with respect to a replete subcategory \(\mathcal{W}\subseteq\mathcal{S}\) if \(\Yoneda(X)\) is compactly supported at \(\mathcal{W}\). This is by Yoneda's lemma equivalent to \(X\in\mathcal{W}^{\perp}\). Here we are employing the following notation of ``Hom-perpendicularity''  for a general additive category \(\mathcal{A}\) and a subcategory \(\mathcal{B}\subseteq\mathcal{A}\):
\begin{itemize}
    \item \(\mathcal{B}^{\perp}\) denotes the full subcategory \(\{A\in \mathcal{A}:\forall B\in\mathcal{B},\Hom_{\mathcal{A}}(B,A)=0\}\),
    \item \({}^{\perp}\mathcal{B}\) denotes the full subcategory \(\{A\in \mathcal{A}:\forall B\in\mathcal{B},\Hom_{\mathcal{A}}(A,B)=0\}\).
\end{itemize}

\subsubsection{Metrics on triangulated categories}
\begin{definition}[{\cite[Definition 1.2]{Neeman25A}}]
Let \(\mathcal{M}=\{B_n\}_{n\in\N}\) be a decreasing chain \(B_1\supseteq B_2 \supseteq B_3\supseteq\ldots\) of full subcategories of \(\mathcal{S}\).
We say that \(\mathcal{M}\) is \textit{a~good metric} on~\(\mathcal{S}\) if for every \(n\in\N\):
\begin{itemize}
    \item \(0\in B_n\),
    \item \(B_n\) is extension-closed, i.e.\ \(Y\in B_n\) whenever \(X,Z\in B_n\) and there exists a~triangle \(X\rightarrow Y\rightarrow Z\rightarrow\Sigma X\) in \(\mathcal{S}\), and
    \item \(\Sigma^{-1} B_{n+1}\cup B_{n+1}\cup \Sigma B_{n+1}\subseteq B_n\).
\end{itemize}
\end{definition}

The categories \(B_n\) of a metric \(\mathcal{M}=\{B_n\}_{n\in\N}\) should be thought as open balls (or open neighbourhoods) shrinking around the zero object. As the cone of every isomorphism is zero, metrics measure how close a morphism is to an isomorphism in the following sense - a morphism \(f:X\rightarrow Y\) in \(\mathcal{S}\) is ``of distance \(\leq \frac{1}{n}\) from an~isomorphism'' if \(\cone f\in B_n\).
With metrics, we can define Cauchy sequences as those sequences whose morphisms are progressively getting shorter.

\begin{definition}[{\cite[Definition~1.6]{Neeman25A},\cite[Remark~3.12]{CummingsGratz24}}]
Equip \(\mathcal{S}\) with a good metric \(\mathcal{M}=\{B_n\}_{n\in\N}\).
Let \(\mathbf{E}=(E_n,e_n)_{n\in\N}\) be a sequence in \(\mathcal{S}\).
We say that \(\mathbf{E}\) \textit{stabilises at} \(n\in\N\) \textit{with respect to} \(\mathcal{M}\) if it stabilises at \(B_n\).

We call \(\mathbf{E}\) \textit{a Cauchy sequence with respect to} \(\mathcal{M}\) is it stabilises at every \(n\in\N\) with respect to \(\mathcal{M}\).

We say that \(\mathbf{E}\) is $n$\textit{-trivial with respect to} \(\mathcal{M}\) for \(n\in\N\) if it is \(B_n\)-trivial.
\end{definition}

\begin{definition}[{\cite[Definition~1.11]{Neeman25A},\cite[Remark~3.12]{CummingsGratz24}}]
Equip \(\mathcal{S}\) with a good metric \(\mathcal{M}=\{B_n\}_{n\in\N}\). Let \(F\in\Mod\dashmodule\mathcal{S}\).
We say that \(F\) is \textit{compactly supported at} \(n\in\N\) \textit{with respect to} \(\mathcal{M}\) if it is compactly supported at \(B_n\). And we call $F$ \textit{compactly supported with respect to} \(\mathcal{M}\) if it is compactly supported at some \(n\in\N\) with respect to \(\mathcal{M}\).
\end{definition}

We will omit the name of the metric with respect to which we are considering the~stabilisation, the triviality, and the compactly supported elements if such a~metric is implicit. And again, any consideration for a sequence \(\mathbf{E}=(E_n,e_n)_{n\in\N}\) to be compactly supported refers to the functor \(\moco \mathbf{E}\), and for an object \(X\in\mathcal{S}\) it refers to the functor \(\Yoneda(X)\).

\begin{remark}
For a sequence \(\mathbf{E}=(E_n,e_n)_{n\in\N}\) to be Cauchy with respect to a~good metric \(\mathcal{M}=\{B_n\}_{n\in\N}\), it means that for every \(M\in\N\) there exists \(N\in\N\) such that for all \(N\leq n\leq m\), the morphism \(e_{n,m}\) is ``of distance \(\leq M\) from an~isomorphism'', i.e.\ \(\cone e_{n,m}\in B_M\).

By \cite[Lemma~1.9]{Neeman25A}, it is enough to test whether a sequence is Cauchy on the~morphisms \(e_n:E_n\rightarrow E_{n+1}\) only.
The sequence \(\mathbf{E}\) is Cauchy if for every \(M\in\N\) there exists \(N\in\N\) such that for all \(N\leq n\) we have \(\cone e_{n}\in B_M\).
\end{remark}

\begin{definition}[{\cite[Definition 1.2]{Neeman25A}}]
    Let \(\mathcal{M}=\{B_n\}_{n\in\N}\) and~\(\mathcal{N}=\{C_n\}_{n\in\N}\) be two metrics on \(\mathcal{S}\).
    We say that \(\mathcal{N}\) is \textit{finer} than \(\mathcal{M}\) if for every \(n\in\N\) there is \(m\in\N\) such that \(C_m\subseteq B_n\), and we denote this by~\(\mathcal{N}\leq\mathcal{M}\).

    We proclaim \(\mathcal{M}\) and \(\mathcal{N}\) to be \textit{equivalent} if
    \(
    \mathcal{M}\leq\mathcal{N}
    \)
    and
    \(
    \mathcal{N}\leq\mathcal{M}
    \).
\end{definition}

Good metrics with balls closed under direct summands have already been studied in \cite{CummingsGratz24}. We propose the following terminology for them.

\begin{definition}
Let \(\mathcal{M}=\{B_n\}_{n\in\N}\) be a good metric on \(\mathcal{S}\).
We say that \(\mathcal{M}\) is \textit{additive} if the ball \(B_n\) is closed under direct summands for all \(n\in\N\).
\end{definition}

By an abuse of notation, we will not distinguish between a triangulated subcategory \(\mathcal{C}\subseteq\mathcal{S}\) and the corresponding \textit{constant metric} \(\{\mathcal{C}\}_{n\in\N}\). Under this identification, additive constant metrics correspond precisely to thick subcategories of~\(\mathcal{S}\).

\subsubsection{Completions of triangulated categories}
\begin{definition}[{\cite[Definition 1.11]{Neeman25A}}]
Equip $\mathcal{S}$ with a~good metric \(\mathcal{M}\).
We define the following full subcategories of \(\Mod\dashmodule\mathcal{T}\) with respect to~\(\mathcal{M}\):
\begin{itemize}
\item the category \(\mathfrak{C}_\mathcal{M}(\mathcal{S})\) of all compactly supported functors from \(\Mod\dashmodule\mathcal{S}\),
\item \textit{the pre-completion} of \(\mathcal{S}\) as
\[\ \ \ \ \ \ \mathfrak{L}_{\mathcal{M}}(\mathcal{S}):=\left\{F\in\Mod\dashmodule\mathcal{S}:\exists \text{ a Cauchy sequence \(\mathbf{E}\) in \(\mathcal{S}\)}, F\simeq\moco\mathbf{E}\right\},\]
\item and \textit{the completion} of \(\mathcal{S}\) as \(\mathfrak{S}_{\mathcal{M}}(\mathcal{S}):=\mathfrak{C}_{\mathcal{M}}(\mathcal{S})\cap\mathfrak{L}_{\mathcal{M}}(\mathcal{S})\).
\end{itemize}
\end{definition}

If there is no space for confusion, we will omit the subscript \(\mathcal{M}\), and write only \(\mathfrak{C}(\mathcal{S})\), \(\mathfrak{L}(\mathcal{S})\), and \(\mathfrak{S}(\mathcal{S})\). And since the (pre)completion and the compactly supported elements are the same for two equivalent metrics, we will also write \(\mathfrak{C}_{\mathfrak{M}}(\mathcal{S})\), \(\mathfrak{L}_{\mathfrak{M}}(\mathcal{S})\), and \(\mathfrak{S}_{\mathfrak{M}}(\mathcal{S})\) for an equivalence class \(\mathfrak{M}\) of good metrics, while meaning the respective categories for any/all metrics from \(\mathfrak{M}\).
Also, for a thick subcategory \(\mathcal{C}\subseteq\mathcal{S}\), the~categories \(\mathfrak{C}_{\mathfrak{C}}(\mathcal{S})\), \(\mathfrak{L}_{\mathfrak{C}}(\mathcal{S})\), and \(\mathfrak{S}_{\mathfrak{C}}(\mathcal{S})\) are meant with respect to \(\mathcal{C}\) interpreted as a constant metric.

The most important fact about metric completions of triangulated categories is that they are triangulated categories themselves.

\begin{theorem}[{\cite[Theorem~2.14]{Neeman25A}}]
If \(\mathcal{S}\) is equipped with a good metric, then \(\mathfrak{S}(\mathcal{S})\) is a triangulated category.
\end{theorem}

The triangulated structure on \(\mathfrak{S}(\mathcal{S})\) is naturally induced by the triangulated structure on \(\mathcal{S}\) - 
the distinguished triangles of \(\mathfrak{S}(\mathcal{S})\) are (compactly supported) module colimits of Cauchy sequences of distinguished triangles from \(\mathcal{S}\), and the~shift functor on \(\mathfrak{S}(\mathcal{S})\) is the restriction of \(\blank\circ\Sigma^{-1}:\Mod\dashmodule\mathcal{S}\rightarrow\Mod\dashmodule\mathcal{S}\) for the shift $\Sigma$ on~\(\mathcal{S}\).

The key idea behind the completions is that compactly supported homological functors from \(\Mod\dashmodule\mathcal{S}\) are precisely those which send ``morphisms close to isomorphisms'' in \(\mathcal{S}\) with respect to the given metric to isomorphisms in \(\Ab\).
Restricting the pre-completion to compactly supported elements then ensures that every morphism in the completion has its well-defined cone. 

This relation between Cauchy sequences and compactly supported elements is encapsulated in the following lemma:

\begin{lemma}[{\cite[Lemma 2.10]{Neeman25A}}]
\label{L:BasicFactorisationProperty}
Assume $\mathcal{S}$ is equipped with a~good metric.
Let \(F\in\Mod\dashmodule\mathcal{S}\) be a cohomological functor compactly supported at \(s\in\N\). Let \(E=\moco \mathbf{E}\in\mathfrak{L}(\mathcal{S})\) for~some Cauchy sequence \(\mathbf{E}=(E_n,e_n)_{n\in\N}\).

If \(\mathbf{E}\) stabilises at $s+1$ from some \(N\in\N\) onwards,
then
\[
\Hom_{\Mod\dashmodule\mathcal{S}}(\varphi_n,F):\Hom_{\Mod\dashmodule\mathcal{S}}(E,F)\rightarrow\Hom_{\Mod\dashmodule\mathcal{S}}\left(\Yoneda(E_n),F\right)
\]
is an isomorphism for all \(n\geq N\) where \(\varphi_n:\Yoneda(E_n)\rightarrow E\) is the~colimit injection. 
\end{lemma}

\subsubsection{Good extensions}
\label{sssec:GoodExtensions}

Calculating completions of the triangulated category \(\mathcal{S}\) directly from the definition as subcategory of \(\Mod\dashmodule\mathcal{S}\) can be rather hard in practice. However, sometimes there exits a computational tool, the so-called good extension. 

\begin{definition}[{\cite[Notation~3.1]{Neeman25A}}]
Let \(\nu:\mathcal{S}\xhookrightarrow{}\mathcal{T}\) be a fully faithful triangulated functor.
By \textit{the restricted Yoneda functor} we mean \(\YonedaKatakana:\mathcal{T}\rightarrow\Mod\dashmodule\mathcal{S}\),
\( X\mapsto\Hom_{\mathcal{T}}\big(\nu(\blank),X\big)\).
\end{definition}

\begin{definition}[{\cite[Definition~3.5]{Neeman18}}]
Assume \(\mathcal{S}\) is equipped with a good metric~\(\mathcal{M}\).
Let \(\nu:\mathcal{S}\xhookrightarrow{}\mathcal{T}\) be a fully faithful triangulated functor, and denote \(\YonedaKatakana\) the~restricted Yoneda functor.
Then $\nu$ is \textit{a good extension with respect to} \(\mathcal{M}\) if:
\begin{itemize}
    \item \(\mathcal{T}\) has countable coproducts, and
    \item the canonical map \(\moco\mathbf{E}\rightarrow\YonedaKatakana\big( \hoco \nu(\mathbf{E})\big)\) in \(\Mod\dashmodule\mathcal{S}\) is an isomorphism for every Cauchy sequence \(\mathbf{E}\) in \(\mathcal{S}\).  
\end{itemize}   
\end{definition}

\begin{definition}[{\cite[Definition~3.11]{Neeman25A}}]
Let \(\nu:\mathcal{S}\xhookrightarrow{}\mathcal{T}\) be a good extension with respect to a good metric \(\mathcal{M}\) on \(\mathcal{S}\).
We define the following full subcategories of \(\mathcal{T}\):
\begin{itemize}
\item \(\mathfrak{C}_\mathcal{M}'(\mathcal{S}):=\YonedaKatakana^{-1}\big(\mathfrak{C}_{\mathcal{M}}(\mathcal{S})\big)\) where \(\YonedaKatakana\) is the~restricted Yoneda functor,
\item \(\mathfrak{L}_{\mathcal{M}}'(\mathcal{S}):=\left\{X\in\mathcal{T}:\exists \text{ a Cauchy sequence \(\mathbf{E}\) in }\mathcal{S}, X\simeq\hoco \nu(\mathbf{E})\right\}\),
\item \(\mathfrak{S}_{\mathcal{M}}'(\mathcal{S}):=\mathfrak{C}_{\mathcal{M}}'(\mathcal{S})\cap\mathfrak{L}_{\mathcal{M}}'(\mathcal{S})\).
\end{itemize}
\end{definition}

Note that by \cite[Observation~3.2]{Neeman25A}, for a metric \(\mathcal{M}=\{B_n\}_{n\in\N}\) we can alternatively describe \(\mathfrak{C}_\mathcal{M}'(\mathcal{S})\) as \(\bigcup_{n\in\N}\nu(B_n)^{\perp}\subseteq\mathcal{T}\).

As in the case of \(\mathfrak{C}(\mathcal{S})\), \(\mathfrak{L}(\mathcal{S})\), and \(\mathfrak{S}(\mathcal{S})\), we will not mention the metric while writing \(\mathfrak{C}'(\mathcal{S})\), \(\mathfrak{L}'(\mathcal{S})\), and \(\mathfrak{S}'(\mathcal{S})\) unless necessary.

If a good extension exists, we can calculate the completion as a triangulated subcategory of the good extension. 

\begin{theorem}[{\cite[Theorem 3.23]{Neeman25A}}]
\label{T:GoodExtensionComputationalTool}
Let \(\mathcal{S}\xhookrightarrow{}\mathcal{T}\) be a good extension with respect to a good metric on \(\mathcal{S}\).
Then the restricted Yoneda functor \(\YonedaKatakana\) determines a~well\=/defined triangulated equivalence \(\YonedaKatakana\restriction_{\mathfrak{S}'(\mathcal{S})}:\mathfrak{S}'(\mathcal{S})\rightarrow\mathfrak{S}(\mathcal{S})\).
\end{theorem}

The prototypical example of a good extension is the inclusion of compact elements.

\begin{theorem}[{\cite[Example~3.9]{Neeman25A},\cite[Lemma~2.8]{Neeman96}}]
\label{T:PrototypicalGoodExtension}
Let \(\mathcal{T}\) be a triangulated category with coproducts. Then the inclusion \(\mathcal{T}^c\xhookrightarrow{}\mathcal{T}\) serves as a good extension with respect to any possible choice of a good metric on \(\mathcal{T}^c\).
\end{theorem}

\subsection{{\normalfont{\textit{t}}}-structures and approximation sequences}

The notion of a $t$-structure on a triangulated category was introduced in \cite[Définition~1.3.1]{BeilinsonBernsteinDeligne82}, see \cite{BeilinsonBernsteinDeligneGabber18} for a~republished version.
In this subsection, we recall its definition and the construction of approximation sequences for a Cauchy sequence. 

\begin{definition}
Let \(\left(\mathcal{S}^{\leq0},\mathcal{S}^{>0}\right)\) be a pair of full subcategories of a triangulated category \(\mathcal{S}\). Assume that \(\mathcal{S}^{\leq0}\) and \(\mathcal{S}^{>0}\) are closed under direct summands. We fix a notation \(\mathcal{S}^{\leq n}:=\Sigma^{-n}\mathcal{S}^{\leq0}\) and \(\mathcal{S}^{> n}:=\Sigma^{-n}\mathcal{S}^{>0}\) for all \(n\in\Z\).

We call \(\left(\mathcal{S}^{\leq0},\mathcal{S}^{>0}\right)\) a \textit{t-structure} on \(\mathcal{S}\) if:
\begin{itemize}
    \item \(\mathcal{S}^{>0}\subseteq(\mathcal{S}^{\leq0})^{\perp}\),
    \item for every \(Y\in\mathcal{S}\) there exist \(X\in\mathcal{S}^{\leq0}\), \(Z\in\mathcal{S}^{>0}\), and a distinguished triangle
    \[X\rightarrow Y\rightarrow Z\rightarrow\Sigma X,\]
    \item and \(\mathcal{S}^{\leq-1}\subseteq\mathcal{S}^{<0}\).
\end{itemize}

In this case, the subcategory \(\mathcal{S}^{\leq0}\) is called \textit{an aisle}, and \(\mathcal{S}^{>0}\) is called \textit{a~coaisle}.
\end{definition}

For a $t$-structure \(\left(\mathcal{S}^{\leq0},\mathcal{S}^{>0}\right)\) on \(\mathcal{S}\) and \(Y\in\mathcal{S}\), the distinguished triangle
    \[X\rightarrow Y\rightarrow Z\rightarrow\Sigma X\]
with \(X\in\mathcal{S}^{\leq0}\) and  \(Z\in\mathcal{S}^{>0}\) is unique up to an isomorphism and is usually called \textit{the approximation triangle}. This name is justified by the fact that \(X\rightarrow Y\) is an~\(\mathcal{S}^{\leq0}\)\=/cover of $Y$ (or a minimal right \(\mathcal{S}^{\leq0}\)-approximation). The~\(\mathcal{S}^{\leq0}\)\=/cover given by an approximation triangle is functorial, we denote this \textit{truncation functor} as \(\tr^{\leq 0}\). 
Dually \(Y\rightarrow Z\) is an~\(\mathcal{S}^{>0}\)-envelope of $Y$ (a minimal left \(\mathcal{S}^{>0}\)\=/approximation) with a corresponding functor  \(\tr^{>0}\). The same approximation theory works for the~shifts of the truncation functors \(\tr^{\leq n}\) and \(\tr^{>n}\).

For the theory of metric completions, it is important to consider approximation sequences of (primarily Cauchy) sequences.
\begin{definition}[{\cite[Definition~4.4]{CummingsGratz24}}]
Let \(\left(\mathcal{S}^{\leq0},\mathcal{S}^{>0}\right)\) be a $t$-structure on a triangulated category \(\mathcal{S}\), and let \(\mathbf{E}=(E_n,e_n)_{n\in\N}\) be a sequence in \(\mathcal{S}\). Fix \(m\in\Z\).

We call the sequence \(\tr^{\leq m}\mathbf{E}=\left(\tr^{\leq m}E_n,\tr^{\leq m}e_n\right)_{n\in\N}\) \textit{an} \(\mathcal{S}^{\leq m}\)\textit{-approximation} of \(\mathbf{E}\) and the sequence \(\tr^{> m}\mathbf{E}=\big(\tr^{> m}E_n,\tr^{> m}e_n\big)_{n\in\N}\) \textit{an} \(\mathcal{S}^{> m}\)\textit{-approximation} of~\(\mathbf{E}\).
\end{definition}

Even if \(\mathbf{E}\) is a Cauchy sequence with respect to some good metric, the approximation sequences do not have to be Cauchy in general. We can see some positive cases, however, in \cite[Section~4]{CummingsGratz24}.

\subsection{Derived categories}

Let $R$ be a ring.
The category of right $R$-modules is denoted \(\ModR\), the notation \(\modf\dashmodule R\) is reserved for the~category of finitely presented right $R$-modules. The category of finitely presented indecomposable modules is denoted \(\ind\dashmodule R\).

We will mainly work with two triangulated categories associated to $R$, namely \(\derived(R)=\derived(\ModR)\), the derived category of $R$, and \(\derived^b(\modf\dashmodule R)\), the bounded derived category of finitely presented $R$-modules. 
The category of compact elements \(\derived(R)^c\subseteq\derived(R)\) is equal to the category of perfect complexes \(\perfect(R)\subseteq\derived(R)\) over $R$, i.e.\ complexes quasi-isomorphic to bounded complexes of finitely generated projective $R$-modules.
The category \(\perfect(R)\) is thus also equal to the essential image of the embedding \(\HomotopyCategory^b(\proj\dashmodule R)\xhookrightarrow{}\derived(R)\) of the bounded homotopy category of finitely generated projective $R$-modules.

By~Theorem~\ref{T:PrototypicalGoodExtension}, the inclusion \(\perfect(R)\xhookrightarrow{}\derived(R)\) is a good extension with respect to all possible good metrics on~\(\perfect(R)\).
This good extension is a strong computational tool for completions of \(\perfect(R)\), giving us \(\mathfrak{S}'\big(\perfect(R)\big)\subseteq\derived(R)\). We can even say more for $R$ hereditary.

\begin{proposition}[{\cite[Proposition~3.1]{Matousek26A}}]
\label{P:BoundedHereditaryCompletion}
Assume $R$ is right hereditary. Then for every good metric on \(\perfect(R)\), the completion \(\mathfrak{S}'\big(\perfect(R)\big)\subseteq\derived(R)\) is a triangulated subcategory of the bounded derived category of (all) $R$-modules \(\derived^b(\ModR)\).
\end{proposition}

\textit{The~standard t-structure} \(\left(\derived(R)^{\leq0},\derived(R)^{>0}\right)\)
on \(\derived(R)\) is determined by the~coho\-mol\-o\-gy functor \(H^0\simeq\Hom_{\derived(R)}(R,\blank):\derived(R)\rightarrow\Ab\) via
\[
\derived(R)^{\leq0}:=\left\{X\in\derived(R):\forall i\in\Z,i>0\Rightarrow H^i(X)=0\right\}
.\]
If $R$ is noetherian and of finite global dimension, then \(\derived^b(\modf\dashmodule R)\simeq\perfect(R)\), the standard $t$-structure on \(\derived(R)\) restricts to a $t$-structure on \(\derived^b(\modf\dashmodule R)\), and the~inclusion \(\derived^b(\modf\dashmodule R)\xhookrightarrow{}\derived(R)\) becomes a good extension with respect to all possible good metrics on \(\derived^b(\modf\dashmodule R)\).

Assume for the rest of the subsection that $R$ is noetherian and hereditary.
The~noetherian condition ensures that \(\modf\dashmodule R\) is an abelian category, and we may consider its wide subcategories. \textit{A wide} subcategory of an abelian category can be characterised as a subcategory closed under taking extensions, kernels, and cokernels (see {\cite[Section~1]{Hovey01},\cite[Lemma~4.2]{Bruning07}}). 
In particular, a wide subcategory is closed under images and direct summands.
By \(\wide(\mathcal{C})\), we will denote the smallest wide subcategory of \(\modf\dashmodule R\) containing \(\mathcal{C}\subseteq\modf\dashmodule R\).

By Brüning's \cite[Theorem~5.1]{Bruning07}, the assignment \(\mathcal{C} \mapsto H^0(\mathcal{C})\) estabilishes a bijection 
\[
\{\text{thick subcategories of }\derived^b(\modf\dashmodule R)\} \leftrightarrow \{\text{wide subcategories of }\modf\dashmodule R
\}
.\]
As $R$ is hereditary, every object \(X\in\derived(R)\) is quasi-isomorphic to its cohomology \(X\simeq\coprod_{i\in\Z}\Sigma^{-i}H^i(X)\). Thus the preimage of a wide subcategory \(\mathcal{W}\subseteq\modf\dashmodule R\) under the aforementioned bijection is 
\(\thick(\mathcal{W})\subseteq\derived^b(\modf\dashmodule R)\).

\subsection{Commutative algebra}

Let $R$ be a commutative ring.
The \textit{Zariski topology} on the prime spectrum \(\Spec(R)\) is the topology where closed sets are of the form \(V(I)=\{\mathfrak{p} \in \Spec(R) : I \leq \mathfrak{p} \} \) for ideals \( I \leq R \).
A subset of \(\Spec(R)\) is \textit{specialisation closed} if it is a union of closed sets.
The topological space \(\Spec(R)\) is connected if and only if $R$ is connected as a ring.

For every prime ideal \(\mathfrak{p}\in\Spec(R)\), we denote \(k(\mathfrak{p}):=R_{\mathfrak{p}}/\mathfrak{p}R_{\mathfrak{p}}\)  its \textit{residue field}. A~finitely generated $R$-module \(M\) is \textit{supported} at \(\mathfrak{p}\in\Spec(R)\) if its localisation \(M_{\mathfrak{p}}\) at \(\mathfrak{p}\) is non-zero. For an arbitrary complex \(X\in\derived(R)\), we the employ the generalisation of \textit{a support} due to Foxby \cite{Foxby79} 
\(\Supp(X):=\left\{\mathfrak{p}\in\Spec(R):X\otimes^\mathbf{L}_Rk(\mathfrak{p})\neq0\right\}\). 
For a~class \(\mathcal{C}\subseteq\derived(R)\), we set \(\Supp(\mathcal{C}):=\bigcup_{C\in\mathcal{C}}\Supp(C)\).
A subset \(\Phi\subset\Spec(R)\) is called \textit{coherent} \cite{Krause08} if for every \(X\in\derived(R)\) it holds that \(\Supp(X)\subseteq\Phi\) if and only if \(\Supp\left(\coprod_{i\in\Z}H^i(X)\right)\subseteq\Phi\).

\begin{lemma}
\label{L:ProjectiveSupportIsLocallyConstant}
Let \(P\) be a finitely generated projective $R$-module.
Then the rank function \(\Spec(R)\rightarrow\N_0\), \(\mathfrak{p}\mapsto \dim_{k(\mathfrak{p})} P\otimes_Rk(\mathfrak{p})\) is locally constant.

In particular, if \(\Phi\) is a connected component of \(\Spec(R)\) and \(\Supp(P)\cap\Phi\neq\emptyset\), then \(\Phi\subseteq\Supp(P)\).
\end{lemma}

\begin{proof}
This is the implication ``a) $\Rightarrow$ c)'' in \cite[Chapitre II, \S 5, \textnumero2, Théorème~1]{Bourbaki06}.   
\end{proof}

Our main focus in this paper is on hereditary noetherian rings.
If the commutative ring $R$ is such, it decomposes as a finite direct product \(R\simeq \prod_{n=1}^ND_n\) of Dedekind domains \(D_1,\ldots,D_N\) for some \(N\in\N\); see for example \cite[Corollary~5.5]{Ando24}.
Lemma~\ref{L:ProjectiveSupportIsLocallyConstant} then tell us that \(\Spec(D_1),\ldots,\Spec(D_N)\) are precisely the~connected components of \(\Spec(R)\).
Therefore, we can often reduce our reasoning to the case of Dedekind domains.

There are several equivalent definitions of \textit{a Dedekind domain}; one such characterisation is that it is a hereditary integral domain. In particular, every principal ideal domain is a Dedekind domain.
Some authors do not allow a field to be considered a~Dedekind domain; however, we do not impose such a restriction.

For the rest of this subsection, we fix a Dedekind domain $D$.
The structure of finitely generated $D$-modules is well-understood. 

\begin{theorem}
\label{T:DedekindFinitelyGeneratedModules}
Every finitely generated $D$\=/module is a direct sum of a projective module $P$ and a torsion module $T$ of the~form \(T\simeq \bigoplus_{i=1}^nR/\mathfrak{p}_i^{l_i}\) for some \(n\in\N_0\), \(\mathfrak{p}_1,\ldots,\mathfrak{p}_n\in\Spec(D)\), and \(l_1,\ldots,l_n\in\N\). This decomposition is unique up to isomorphisms and reordering of the summands.
\end{theorem}

Thus the support of a finitely generated module $M$ over $D$ is either the whole \(\Spec(D)\) if $M$ contains a non-zero projective direct summand, or finite. 

Let $P$ and $Q$ be indecomposable finitely generated projective $D$-modules. Then $P$ is generated by either one element (and is isomorphic to $D$) or two elements, and it holds that \(P\oplus Q\simeq D\oplus (P\otimes_{D}Q)\). This makes the tensor product \(\otimes_D\) into a~commutative group operation on \textit{the Picard group} \(\Pic(D)\) of $D$, which consists of equivalence classes of indecomposable finitely generated projective $D$-modules up isomorphisms and free direct summands. The structure of indecomposable finitely generated projectives can be rather complicated; in fact, any abelian group is the~Picard group of some Dedekind domain \cite{Claborn66}.

\subsection{Finite dimensional algebras}
Let $K$ be an algebraically closed field. By Drozd's trichotomy, every finite dimensional $K$-algebra is either of representation-finite type, tame, or wild. Every connected hereditary finite dimensional $K$-algebra is Morita equivalent to a path algebra $KQ$ for some finite acyclic connected quiver~$Q$, and we can determine its representation type based on the underlying diagram of the quiver~$Q$ - the algebra is representation-finite if and only if the diagram is simply laced Dynkin (type \(\mathbb{A},\mathbb{D}\), and \(\mathbb{E}\)); and it is tame if and only if the diagram is Euclidean \big(or extended Dynkin - type \(\tilde{\mathbb{A}},\tilde{\mathbb{D}}\), or \(\tilde{\mathbb{E}}\)\big).

In this subsection, we fix an Euclidean quiver $Q$ with \(N\in\N\) verticies, and we shall consider the tame algebra \(KQ\). Both \(\modf\dashmodule KQ\) and \(\mathcal{S}:=\derived^b(\modf\dashmodule KQ)\) are $K$-linear Krull-Schmidt categories.

Consider the Auslander-Reiten translation \(\tau:\mathcal{S}\rightarrow\mathcal{S}\) (see e.g.\ \cite{ReitenVanderbergh02}).
The category \(\ind\dashmodule KQ\) (up to isomorphisms) consists of three components:
\begin{itemize}[leftmargin=5mm] 
    \item the preprojectives \[
    \mathcal{P}:=\left\{M\in\ind\dashmodule KQ:\exists P\in\ind\dashmodule KQ, P\text{ projective},\exists n\in\N_0,M\simeq\tau^{-n}P\right\}
    ,\]
    \item the preinjectives \begin{align*}
    \mathcal{Q}&:=\left\{M\in\ind\dashmodule KQ:\exists I\in\ind\dashmodule KQ, I\text{ injective},\exists n\in\N_0,M\simeq\tau^{n}I\right\}\\
    &=\left\{M\in\ind\dashmodule KQ:\exists P\in\ind\dashmodule KQ, P\text{ projective},\exists n\in\N,M\simeq\tau^{n}\Sigma P\right\}
    \text{, and}\end{align*}
    \item the regular modules \(\mathcal{R}\).
\end{itemize}

The morphisms in \(\ind\dashmodule KQ\) go only in one direction, meaning that \(\mathcal{P}\subseteq\mathcal{R}^{\perp}\) and \(\mathcal{P},\mathcal{R}\subseteq\mathcal{Q}^{\perp}\) in \(\modf\dashmodule KQ\).

The regular component \(\coprodf(\mathcal{R})=\wide(\mathcal{R})\simeq \coprod_{i\in\ProjectiveLine{K}} \mathbf{t}_i\) decomposes as a~coproduct of so-called tubes indexed by the projective line \(\ProjectiveLine{K}\).
All the tubes are mutually \textit{orthogonal}, i.e.\ \(\mathbf{t}_i,\Sigma\mathbf{t}_i\subseteq \mathbf{t}_j^\perp\) in \(\mathcal{S}\) for all \(i\neq j\in\ProjectiveLine{K}\).
Each tube \(\mathbf{t}\) is an orbit under the autoequivalence \(\tau\), and the minimal \(n\in\N\) such that \(\tau^n\) acts as the identity on \(\mathbf{t}\) is called \textit{the rank of} \(\mathbf{t}\). Tubes of rank $1$ are \textit{homogenous}. All tubes have finite rank, and at most $3$ tubes can be non-homogenous.  
By~\cite[Proposition~2.4.2]{Dichev09}, every tube of rank \(n\in\N\) has \(\binom{2n}{n}\) wide subcategories.

\begin{definition}
Let \(M\in\modf\dashmodule KQ\). Then $M$ is called \textit{rigid} if \(\Ext^1_{KQ}(M,M)=0\). A rigid module $M$ is called \textit{exceptional} if \(\End_{KQ}(M)\simeq K\).

Let \(n\in\N\). An $n$-tuple \((X_1,\ldots,X_n)\) of finitely generated $KQ$-modules is called an \textit{exceptional sequence of length n} if it consisits of exceptional objects, and for all \(1\leq i<j\leq n\) we have \(\Hom_{KQ}(X_j,X_i)\simeq0\simeq\Ext^1_{KQ}(X_j,X_i)\).
If $n$ equals the~number of verticies of $Q$, then the exceptional sequence is called \textit{complete}.
\end{definition}

Any complete exceptional sequence \((X_1,\ldots,X_N)\) generates the module category \(\modf\dashmodule KQ\) in the~sense that \(\wide(X_1,\ldots,X_N)=\modf\dashmodule KQ\).
Furthermore, the category \(\wide(X_1,\ldots,X_n)\) for an exceptional sequence of length \(n\in\N\) is equivalent to \(\modf\dashmodule KQ'\) for some finite acyclic quiver $Q'$ with $n$ verticies, and the embedding \(\modf\dashmodule KQ'\xhookrightarrow{}\modf\dashmodule KQ\) is exact and induces isomorphisms on Hom and Ext. 
Additionally, the category \(\wide(X_1,\ldots,X_n)^{\perp}\subseteq\modf\dashmodule KQ\) possesses an exceptional sequence of length \(N-n\), and \(\wide(X_1,\ldots,X_n)={}^{\perp}\big(\wide(X_1,\ldots,X_n)^{\perp}\big)\) (see \cite{Boevey93} for details).

A module \(S\in\mathcal{R}\) without proper non-zero regular submodules is \textit{simple regular}. Every tube of rank \(n\in\N\) contains precisely $n$ different simple regulars up to isomorphism.

\begin{lemma}
\label{L:SincereModule}
Let \(S\in\ind\dashmodule KQ\) be simple regular from a homogenous tube. Then \(\Sigma\mathcal{P}\cap S^{\perp}=\emptyset=\mathcal{Q}\cap S^{\perp}\) in \(\mathcal{S}\).
\end{lemma}

\begin{proof}
The notation \((\blank)^*=\Hom_K(\blank,K)\) shall denote the standard $K$-dual.
Since \(S\) simple regular lies in a homogenous tube, then its dimension vector is sincere (see e.g.\ \cite[Corollary~XI.3.9]{SimsonSkowronski07}), meaning \(\Hom(P,S)\neq0\) for every projective \(P\in\ind\dashmodule KQ\).
Any preinjective is of the form \(\tau^n\Sigma P\) for some \(n\in\N\) and \(P\) indecomposable projective. 
Using that \(\tau\Sigma:\mathcal{S}\rightarrow\mathcal{S}\) is a Serre functor on \(\mathcal{S}\) (see \cite[Theorem~II.1.3]{ReitenVanderbergh02}), we get
\begin{align*}
0&\neq\Hom_{\mathcal{S}}(P,S)\simeq\Hom_{\mathcal{S}}(P,\tau^{-n-1}S)\simeq \Hom_{\mathcal{S}}(\tau^{-n-1}S,\tau\Sigma P)^*\\
&\simeq \Hom_{\mathcal{S}}(S,\tau^n\Sigma P)^*
.\end{align*}
The proof that \(\Hom_{\mathcal{S}}(S,\Sigma P')\neq0\) for all \(P'\in\mathcal{P}\) is dual.
\end{proof}

We intend to use Brüning's correspondence to understand thick subcategories of~\(\mathcal{S}\). We will denote \(\mathfrak{R}:=\thick(\mathcal{R})\simeq\coprod_{i\in\ProjectiveLine{K}}\coprod_{j\in\Z}\Sigma^j\tube_i\subseteq\mathcal{S}\) the thick subcategory generated by all the regular tubes.

\begin{theorem}[{\cite[Theorem~3.2.15]{Dichev09},\cite[Proposition~6.14]{Kohler11}}]
\label{T:RegularOrExceptional}
Let \(\mathcal{C}\) be a thick subcategory of~\(\mathcal{S}\). Then at least one of the following conditions holds:
\begin{enumerate}[label=(\roman*), leftmargin=10mm]
\item \(\mathcal{C}=\thick(X_1,\ldots,X_n)\) for an exceptional sequence \((X_1,\ldots,X_n)\) in \(\modf\dashmodule KQ\),
\item \(\mathcal{C}\subseteq\mathfrak{R}\). 
\end{enumerate}
\end{theorem}

\begin{lemma}
\label{L:InfiniteChainDiffersAtHomogenousTubes}
There exists a number \(M\in\N\) (depending on $Q$) such that for any strictly decreasing chain \(\mathcal{C}_1\supsetneq\mathcal{C}_2\supsetneq\cdots\supsetneq\mathcal{C}_M\) of thick subcategories of~\(\mathcal{S}\) we can find indices \(1\leq i<j\leq M\) and a homogenous regular tube \(\tube\subseteq\modf\dashmodule KQ\) with \(\mathcal{C}_i,\mathcal{C}_j\subseteq\mathfrak{R}\), \(\tube\subseteq\mathcal{C}_i\) and \(\tube\cap\,\mathcal{C}_j=0\).
\end{lemma}

\begin{proof}
Let \(\mathfrak{U}\subseteq\mathfrak{R}\) be the thick subcategory generated by all the non-homogenous tubes, and \(\mathfrak{V}\subseteq\mathfrak{R}\) be the thick subcategory generated by all the homogenous tubes.
Every thick subcategory of \(\mathfrak{R}\) is of the form \(\mathcal{U}\oplus\mathcal{V}\) for some thick subcategories \(\mathcal{U}\subseteq\mathfrak{U}\) and \(\mathcal{V}\subseteq\mathfrak{V}\).
Since there are only finitely many non-homogenous tubes, and each tube has only finite number of wide subcategories, there are only finitely many thick subcategories of \(\mathfrak{U}\); let us denote this number by $M'$.

Set \(M:=N+M'+2\).
Let \(\mathcal{C}_1\supsetneq\mathcal{C}_2\supsetneq\cdots\supsetneq\mathcal{C}_M\) be a strictly decreasing chain of thick subcategories of \(\mathcal{S}\). 
If \(\mathcal{C}_i\) and \(\mathcal{C}_j\) for some \(1\leq i<j\leq M\) are generated by exceptional sequences, then the exceptional sequence generating \(\mathcal{C}_j\) must be strictly shorter than the exceptional sequence generating \(\mathcal{C}_i\) because the~equality would yield \(\mathcal{C}_i=\mathcal{C}_j\). 
Therefore, at most \(N+1<M\) thick subcategories in the~chain can be generated by an exceptional sequence. Theorem~\ref{T:RegularOrExceptional} then tells us that the~remaining thick subcategories must be then generated by regular elements.

We restrict ourselves to a sub-chain \(\mathcal{D}_1\supsetneq\mathcal{D}_2\supsetneq\cdots\supsetneq\mathcal{D}_{M'+1}\) of thick subcategories of \(\mathcal{S}\) with \(\mathcal{D}_i\subseteq\mathfrak{R}\) for all \(1\leq i\leq M'+1\).
By the pigeonhole principle, there exist \(1\leq i< j\leq M'+1\) and thick subcategories \(\mathcal{U}\subseteq\mathfrak{U}\) and \(\mathcal{V},\mathcal{W}\subseteq\mathfrak{V}\) such that \(\mathcal{D}_i=\mathcal{U}\oplus\mathcal{V}\) and \(\mathcal{D}_j=\mathcal{U\oplus\mathcal{W}}\). This forces \(\mathcal{W}\subsetneq\mathcal{V}\), and these two subcategories must differ by at least one homogenous tube.
\end{proof}

\section{Constant metrics, smashing subcategories, and ring epimorphisms}
\label{sec:smash}

We fix a ring $A$ across this whole section, and denote \(\mathcal{S}:=\perfect(A)\). We shall make use of the fact that the inclusion \(\mathcal{S}\xhookrightarrow{}\derived(A)\) is a good extension, so all the~categories of pre-completions, completions, and compactly supported elements will be calculated inside \(\derived(A)\).

In this short section, we intend to relate the completion \(\mathfrak{S}'(\mathcal{S})\subseteq\derived(A)\) with respect to a constant metric to a subcategory \(\derived(B)\subseteq\derived(A)\) determined by a~homological ring epimorphism \(A\rightarrow B\) (see Proposition~\ref{P:HomologicalEpimorphismDeterminingCompletion} and Corollary~\ref{C:CompletionUniversalLocalisation}).

Recall that if \(\mathcal{T}\) is a category with coproducts, then a triangulated subcategory of \(\mathcal{T}\) is \textit{localising} if it is closed under (infinite) coproducts. For \(\mathcal{C}\subseteq\mathcal{T}\) a class, the~smallest localising subcategory of \(\mathcal{T}\) containing \(\mathcal{C}\) is called \textit{a localising subcategory generated by} \(\mathcal{C}\) and denoted \(\Loc(\mathcal{C})\). A localising subcategory \(\mathcal{L}\subseteq\mathcal{T}\) is called \textit{smashing} if the inclusion \(\mathcal{L}\xhookrightarrow{}\mathcal{T}\) admits a right adjoint and \(\mathcal{L}^{\perp}\) is closed under coproducts.
Smashing subcategories give rise to recollements of triangulated categories (defined in \cite[Subsection~1.4]{BeilinsonBernsteinDeligne82}).

\begin{lemma}
\label{L:SmashingRecollement}
Let \(\mathcal{T}\) be a compactly generated triangulated category.
Consider the~good extension \(\mathcal{T}^c\xhookrightarrow{}\mathcal{T}\).
Let \(\mathcal{C}\subseteq\mathcal{T}^c\) be a triangulated subcategory. Then there exists a recollement diagram 
\begin{equation}\nonumber   
\xymatrix@C=0.5cm{\mathfrak{C}'_{\mathcal{C}}(\mathcal{T}^c) \ar[rrr]^{i_*} &&& \mathcal{T} \ar[rrr]^{j^*}  \ar @/_1.5pc/[lll]_{i^*}  \ar @/^1.5pc/[lll]_{i^!} &&& \Loc(\mathcal{C})\ar @/_1.5pc/[lll]_{j_!} \ar @/^1.5pc/[lll]_{j_*} } 
\end{equation}
with \(i_*:\mathfrak{C}'_{\mathcal{C}}(\mathcal{T}^c)\xhookrightarrow{}\mathcal{T}\) and \(j_!:\Loc(\mathcal{C})\xhookrightarrow{}\mathcal{T}\) being the respective inclusions.
\end{lemma}

\begin{proof}
\(\Loc(\mathcal{C})\) is a smashing subcategory of \(\mathcal{T}\) by \cite[Corollary~2.8]{HugelMarksVitoria19}. The existence of the recollement diagram above then follows from \cite[Proposition~2.9]{HugelMarksVitoria19} (see also \cite[Chapter~4]{Nicolas07}) where we identify \(\Loc(\mathcal{C})^{\perp}=\mathcal{C}^\perp=\mathfrak{C}_{\mathcal{C}}'(\mathcal{T}^c)\) using \cite[Lemma~3.8]{Neeman25A}. 
\end{proof}

Although every localising subcategory of 
\(\derived(A)\) (more generally, of a~compactly generated triangulated category) generated by perfect complexes (compact objects) is smashing, the converse is the assertion of the so-called \textit{telescope conjecture}. While there exist counterexamples to the conjecture (e.g.\ \cite[Example~7.8]{KrauseStovicek10}) in its full generality, it has been proven for some classes of rings, such as hereditary rings by Krause and Šťovíček \cite[Theorem~A]{KrauseStovicek10} and commutative noetherian rings by Neeman \cite{Neeman92A}.

\begin{theorem}[{\cite[Theorem~3.3, Corollary~3.4]{Neeman92A}}]
\label{T:TelescopeConjectureCommutativeNoetherian}
Let $A$ be a commutative noetherian ring.
Then there are bijections between:
\begin{itemize}
    \item thick subcategories of \(\perfect(A)\),
    \item smashing subcategories of \(\derived(A)\), and
    \item specialisation closed subsets of \(\Spec(A)\).
\end{itemize}
More specifically, to a thick subcategory \(\mathcal{C}\) of \(\perfect(A)\) we assign the localising subcategory \(\Loc(\mathcal{C})\) of \(\derived(A)\) and the specialisation closed subset \(\Supp(\mathcal{C})\), respectively.
\end{theorem}

Sometimes, the validity of the telescope conjecture also provides us for each smashing subcategory of \(\derived(A)\) with a corresponding \textit{homological epimorphism}; that is a ring epimorphism \(A\rightarrow B\) such that \(B\otimes_AB\simeq B\) and \(\Tor^A_n(B,B)=0\) for all \(n\in\N\), or equivalently (via \cite[Theorem~4.4]{GeigleLenzing91}), such that the induced functor \(\derived(B)\rightarrow\derived(A)\) is fully faithful.

\begin{proposition}
\label{P:HomologicalEpimorphismDeterminingCompletion}
Let \(\mathcal{C}\subseteq\mathcal{S}\) be a thick subcategory.
Assume further that at least one of these conditions holds:
\begin{enumerate}[label=(\roman*)]
\item $A$ is hereditary, or
\item $A$ is commutative noetherian, and \(\Supp(\mathcal{C})\) has a coherent complement in~\(\Spec(A)\).
\end{enumerate}
Then there exists a~homological epimorphism \(A\rightarrow B\) satisfying \(\mathfrak{S}_{\mathcal{C}}'(\mathcal{S})\subseteq\perfect(B)\) where we identify \(\derived(B)\) with \(\im(i_*)\) from the recollement diagram
\begin{equation}\nonumber   
\xymatrix@C=0.5cm{\derived(B) \ar[rrr]^{i_*} &&& \derived(A) \ar[rrr]^{j^*}  \ar @/_1.5pc/[lll]_{\blank\otimes^\mathbf{L}_AB}  \ar @/^1.5pc/[lll]_{i^!} &&& \Loc(\mathcal{C})\ar @/_1.5pc/[lll]_{j_!} \ar @/^1.5pc/[lll]_{j_*} } 
\end{equation}
with \(j_!:\Loc(\mathcal{C})=\Ker\left(\blank\otimes^\mathbf{L}_AB\right)\xhookrightarrow{}\derived(A)\) and \(i^!=\mathbf{R}\Hom_A(B,\blank)\).

Furthermore, the equality \(\mathfrak{S}_{\mathcal{C}}'(\mathcal{S})=\perfect(B)\) holds if and only if \(B\in\mathfrak{S}_{\mathcal{C}}'(\mathcal{S})\) and \(\mathfrak{S}_{\mathcal{C}}'(\mathcal{S})\) is idempotent complete.
\end{proposition}
    
\begin{proof}
We start by obtaining a recollement diagram from Lemma~\ref{L:SmashingRecollement} applied to the good extension \(\mathcal{S}\xhookrightarrow{}\derived(A)\) and the subcategory \(\mathcal{C}\).

If ``i)'' holds, then the~desired homological epimorphism \(A\rightarrow B\) such that \(\Loc(\mathcal{C})=\Ker\left(\blank\otimes^\mathbf{L}_AB\right)\) and \(\derived(B)\simeq\Loc(\mathcal{C})^\perp\)
exists
by \cite[Section~3, Theorem~8.1]{KrauseStovicek10}.

Now, assume ``ii)'' holds.
By the bijection from Theorem~\ref{T:TelescopeConjectureCommutativeNoetherian}, the category \(\Loc(\mathcal{C})\) consists of objects of \(\derived(A)\) supported at \(\Supp(\mathcal{C})\). 
As \(\Supp(\mathcal{C})\) has a coherent complement, \cite[Theorem~4.9, Corollary~4.10]{HugelMarksStovicekTakahashiVitoria20} yields the existence of a~homological (even flat) epimorphism \(A\rightarrow B\) again satisfying \(\Loc(\mathcal{C})=\Ker\left(\blank\otimes^\mathbf{L}_AB\right)\) and \(\derived(B)\simeq\Loc(\mathcal{C})^\perp\).

Regardless of whether ``i)'' or ``ii)'' happens,
it follows for the completion of \(\mathcal{S}\) that \(\mathfrak{C}'_{\mathcal{C}}(\mathcal{S})=\Loc(\mathcal{C})^\perp=\derived(B)\), so \(\mathfrak{S}'_{\mathcal{C}}(\mathcal{S})\subseteq\derived(B)\).

Let \(E=\hoco \mathbf{E}\in\mathfrak{S}'_{\mathcal{C}}(\mathcal{S})\) for a Cauchy sequence \(\mathbf{E}=(E_n,e_n)_{n\in\N}\). We may w.l.o.g.\ assume \(\cone(e_{n,m})\in\mathcal{C}\) for all \(n\leq m\in\N\).
Since \(\mathcal{C}\subseteq\Ker\left(\blank\otimes^\mathbf{L}_AB\right)\), we get that
\[
E\simeq E\otimes^\mathbf{L}_AB\simeq\left(\hoco \mathbf{E}\right)\otimes^\mathbf{L}_AB\simeq\hoco\left(\mathbf{E}\otimes^\mathbf{L}_AB\right)
\]
is a homotopy colimit of a constant sequence \(\mathbf{E}\otimes^\mathbf{L}_AB\) in \(\derived(B)\) consisting of isomorphisms. Suppose \(E_1\simeq P^\bullet\) for a bounded complex \(P^\bullet\) of finitely generated projective $A$-modules. Then \(E\simeq E_1\otimes^\mathbf{L}_AB\simeq P^\bullet\otimes_AB\) is isomorphic to a bounded complex of finitely generated projective $B$-modules, so \(E\in\perfect(B)\). As $E$ was arbitrary, we get \(\mathfrak{S}_{\mathcal{C}}'(\mathcal{S})\subseteq\perfect(B)\).

The final part of the statement follows directly from \(\perfect(B)=\thick(B)\).
\end{proof}

The ring $B$ above is obtained as the endomorphism ring \(\End_A(i_*i^*A)\) using the~recollement notation of Lemma~\ref{L:SmashingRecollement} and is uniquely determined by the thick subcategory \(\mathcal{C}\subseteq\mathcal{S}\).

A universal localisation is a generalisation of a localisation of a commutative ring at a multiplicative set.

\begin{definition}[{\cite[Theorem~4.1]{Schofield85}}]
Let \(f:A\rightarrow B\) be a ring homomorphism and \(\Phi\) a set of $R$\=/homomorphisms between finitely generated projective $R$-modules.
The~ring homomorphism $f$ is called \(\Phi\)\textit{-inverting} if for all \(\alpha:P\rightarrow Q\) in \(\Phi\), the~$B$\=/module homomorphism \(\alpha\otimes_A\id_B:P\otimes_AB\rightarrow Q\otimes_AB\) is an isomorphism.

The ring homomorphism $f$ (together with the ring $B$) is called a \textit{universal localisation of} $A$ \textit{at} \(\Phi\) if $f$ is \(\Phi\)-inverting, and for every \(\Phi\)-inverting ring homomorphism \(f':A\rightarrow B'\) there exits a unique ring homomorphism \(g:B\rightarrow B'\) such that \(f'=gf\).
\end{definition}

Under favourable conditions, universal localisations and homological epimorphisms coincide, thus allowing us to express metric completions with respect to constant metrics as subcategories of derived categories of universal localisations.  

\begin{corollary}
\label{C:CompletionUniversalLocalisation}
Assume $A$ is a hereditary ring or a commutative noetherian ring of Krull dimension at most~$1$. Let \(\mathcal{C}\subseteq\mathcal{S}\) be a thick subcategory. Then there exists a~universal localisation \(A\rightarrow B\) such that \(\mathfrak{S}_{\mathcal{C}}'(\mathcal{S})\subseteq\perfect(B)\) in the sense of Proposition~\ref{P:HomologicalEpimorphismDeterminingCompletion}.
\end{corollary}
    
\begin{proof}
Under the conditions on $A$, every homological epimorphism \(A\rightarrow B\) is a~universal localisation. This is due to \cite[Theorem~6.1]{KrauseStovicek10} in the hereditary case and due to \cite[Theorem~5.7]{HugelMarksStovicekTakahashiVitoria20} in the case of a commutative noetherian ring of Krull dimension at most~$1$.

The only thing remaining to be checked is that conditions of Proposition~\ref{P:HomologicalEpimorphismDeterminingCompletion} are, in fact, satisfied. For $A$ hereditary, there is nothing to check. 
If $A$ is commutative noetherian of~Krull dimension at most $1$, then all the subsets of \(\Spec(A)\) are coherent by \cite[Corollary~4.3]{Krause08}, including \(\Spec(A)\setminus\Supp(\mathcal{C})\). With that we may conclude the proof.
\end{proof}

Note that if $R$ is hereditary, then \(R_\mathcal{C}\) is also hereditary by \cite[Lemma~1.4]{NeemanRanickiSchofield04}, so \(\perfect(R_{\mathcal{C}})\simeq\derived^b(\modf\dashmodule R_{\mathcal{C}})\).

\begin{notation}
Given a hereditary ring $R$ and a thick subcategory \(\mathcal{C}\subseteq\perfect(R)\), we shall denote the universal localisation determined by \(\mathcal{C}\) in the sense of Corollary~\ref{C:CompletionUniversalLocalisation} as \(R_{\mathcal{C}}\). 
We will then always implicitly treat \(\derived(R_\mathcal{C})\) as a thick subcategory of~\(\derived(R)\) and the~completion \(\mathfrak{S}'_{\mathcal{C}}\big(\perfect(R)\big)\subseteq\derived(R)\) as a triangulated subcategory of \(\derived^b(\modf\dashmodule R_{\mathcal{C}})\simeq\perfect(R_\mathcal{C})\subseteq\derived(R)\) using Proposition~\ref{P:HomologicalEpimorphismDeterminingCompletion}.
\end{notation}

\section{Computational tools for hereditary and noetherian rings}
\label{sec:tools}

This section contains multiple new results about completions of bounded derived categories of certain, mostly noetherian hereditary, rings. Those facts will be useful later for the classification of completions over commutative hereditary noetherian rings (Section~\ref{sec:CommutativeNoetherian}) and hereditary tame algebras (Section~\ref{sec:TameAlgebras}).
Among other things, for a suitable hereditary ring $R$, we show how to construct Cauchy sequences with uniformly bounded cohomology (Corollary~\ref{C:UniformBoundedCohomologyCauchySequenceRestrictionGeneralised}); that \(\mathfrak{S}'\big(\derived^b(\modf\dashmodule R)\big)\) is idempotent complete (Proposition~\ref{P:CompletionIsIdempotentComplete}); and explicitly describe completions of \(\mathfrak{S}'\big(\derived^b(\modf\dashmodule R)\big)\) with respect to constant additive metrics (Proposition~\ref{P:CountablyGeneratedCompletion}).
We also recall the construction of \(\mathcal{B}\)-trivial sequences by Cummings and Gratz in Proposition~\ref{P:CummingsGratzTrivialisationArgument}.

Through this whole section we work with a ring $R$ which is a finitely generated \(\kk\)-algebra over a commutative noetherian ring $\kk$.
Typical choices of \(\kk\) will be \(\kk=K\) for $K$ a field, thus making $R$ into a finite dimensional $K$-algebra, or \(\kk=R\) when $R$ itself is commutative noetherian.

We fix a notation \(\mathcal{S}:=\perfect(R)\). Note that \(\mathcal{S}\simeq\derived^b(\modf\dashmodule R)\) whenever $R$ is of finite global dimension. We also denote \(\mathcal{T}:=\derived(R)\) and note that the inclusion \(\mathcal{S}\xhookrightarrow{}\mathcal{T}\) is a good extension, so all (pre)completions will be calculated inside \(\mathcal{T}\).

\begin{lemma}
\label{L:PickingMaximalSummands}
Let \(0\in\mathcal{C}\) be a replete subcategory of~\(\mathcal{K}:=\HomotopyCategory^b(\proj\dashmodule R)\). Then for any two objects \(X,Y\in\mathcal{K}\) it holds that \(\Hom_{\mathcal{K}}(X,Y)\) is a noetherian \(\kk\)-module.

In particular, every object \(U\in\mathcal{K}\) contains a maximal direct summand \(V\in\mathcal{C}\).
\end{lemma}

\begin{proof}
Let \(X,Y\in\mathcal{K}\). Then the $\kk$-module \(\Hom_{\mathcal{K}}(X,Y)\) is a quotient of the \(\kk\)\=/module \(\Hom_{\Chain^b(\proj\dashmodule R)}(X,Y)\) by null-homotopic maps, while \(\Hom_{\Chain^b(\proj\dashmodule R)}(X,Y)\) is a~$\kk$\=/submodule of \(\End_{R}(F)\) for some finitely generated free $R$-module $F$. As $R$ is a noetherian \(\kk\)-module, \(\Hom_{\mathcal{K}}(X,Y)\) is then a noetherian \(\kk\)-module as well. 

Now let \(U\in\mathcal{K}\).
There definitely exists some direct summand of $U$ belonging to \(\mathcal{C}\), namely the~zero object. If there did not exist a maximal direct summand from $\mathcal{C}$ of $U$, we would be able to use induction to construct an infinite, strictly increasing chain of \(\kk\)\=/submodules of~\(\End_{\mathcal{K}}(U)=\Hom_{\mathcal{K}}(U,U)\), which is impossible in a noetherian module.
\end{proof}

We state a variant of \cite[Lemma~3.17]{CummingsGratz24} for perfect complexes over \(\kk\)-algebras.

\begin{proposition}
\label{P:CummingsGratzTrivialisationArgument}
Let $\mathcal{B}$ and $\mathcal{C}$ be replete subcategories of $\mathcal{S}$ such that $\mathcal{C}$ is closed under extensions and direct summands. Let $\mathbf{E}$ be a sequence in $\mathcal{S}$ that is compactly supported at $\mathcal{B}$. Then there exists a~subsequence $\mathbf{E_I}$ of $\mathbf{E}$ and a $\mathcal{B}$-trivial component $\mathbf{\tilde{E}_I}$ of $\mathbf{E_I}$ such that
    \[
        \moco \mathbf{E} \simeq \moco \mathbf{\tilde{E}_I},
    \]
    and $\mathbf{\tilde{E}_I}$ is compactly supported at $\mathcal{B}$. Moreover, if $\mathbf{E}$ stabilises at $\mathcal{C}$, then so does~$\mathbf{\tilde{E}_I}$.
\end{proposition}

\begin{proof}
The exact same proposition was proved by Cummings and Gratz in \cite[Lemma~3.17]{CummingsGratz24} for \(\mathcal{S}\) being a Krull-Schmidt triangulated category. The Krull-Schmidt property of \(\mathcal{S}\) is used in the beginning of the proof in order to pick a maximal direct summand of $E$ belonging to \(\mathcal{B}\) for any object \(E\in\mathcal{S}\).
However, by Lemma~\ref{L:PickingMaximalSummands} we may perform the same selection of a maximal direct summand in the category \(\mathcal{S}=\perfect(R)\simeq\HomotopyCategory^b(\proj\dashmodule R)\).
Therefore, the original proof of \cite[Lemma~3.17]{CummingsGratz24} also applies \textit{verbatim} as a proof of this proposition.
\end{proof}

\begin{corollary}
\label{C:CummingsGratzTrivialisationForSequences}
Equip \(\mathcal{S}\) with an~additive good metric.
Let \(\mathcal{B}\) be a~replete subcategory of~\(\mathcal{S}\).
If \(\mathbf{E}\) is a Cauchy sequence compactly supported at \(\mathcal{B}\), then there exists a \(\mathcal{B}\)\=/trivial component \(\mathbf{F}\) of a subsequence of~\(\mathbf{E}\) such that \(\mathbf{F}\) is Cauchy and compactly supported at \(\mathcal{B}\), and \(\moco \mathbf{E}\simeq\moco \mathbf{F}\).
\end{corollary}

For hereditary rings, the above construction can be also used to change a generic Cauchy sequence into a Cauchy sequence with all entries uniformly bounded in cohomology.

\begin{corollary}
\label{C:UniformBoundedCohomologyCauchySequenceRestrictionGeneralised}
Assume $R$ is a hereditary ring, and equip \(\mathcal{S}\) with an~additive good metric. Let \(\mathbf{E}\) be a Cauchy sequence.
If \(E:=\hoco \mathbf{E}\in\mathcal{T}\) has cohomology concentrated in degrees \(N,\ldots,M\) for some \(N\leq M\in\Z\), then there exists a Cauchy sequence \(\mathbf{F}=(F_n,f_n)_{n\in\N}\) such that \(\hoco F\simeq E\) and for all \(n\in\N\) the object \(F_n\) has cohomology concentrated in degrees \(N,\ldots,M+1\).
\end{corollary}

\begin{proof}
Let \(\left(\mathcal{S}^{\leq0},\mathcal{S}^{>0}\right)\) be the standard $t$-structure on \(\mathcal{S}\simeq\derived^b(\modf\dashmodule R)\).
The fact that \(E\in\mathcal{T}\) has cohomology concentrated in degrees \(N,\ldots,M\) implies that \(\mathbf{E}\) is compactly supported with respect to the full subcategory \(\mathcal{B}:=\mathcal{S}^{\leq N-1}\cup\mathcal{S}^{>M+1}\).

Since  the conditions of Corollary~\ref{C:CummingsGratzTrivialisationForSequences} are satisfied, there exists an \(\mathcal{B}\)-trivial Cauchy sequence \(\mathbf{F}=(F_n,f_n)_{n\in\N}\) in \(\mathcal{S}\) with \(\hoco \mathbf{F}\simeq E\) holding in \(\mathcal{T}\).
As \(\mathbf{F}\) is \(\mathcal{B}\)\=/trivial, for all \(n\in\N\) the object \(F_n\) does not contain a direct summand from \(\mathcal{B}\). This is equivalent to $F_n$ having cohomology concentrated in degrees \(N,\ldots,M+1\) because \(F_n\) is a coproduct of its cohomologies due to $R$ being hereditary.
\end{proof}

In a series of two lemmas and a proposition, we now aim to prove that \(\mathfrak{S}'(S)\) is idempotent complete for any choice of an additive good metric on \(\mathcal{S}\) whenever $R$ is hereditary.

The proof of the following lemma uses the idea behind Brüning's correspondence between wide categories of a hereditary abelian category and thick subcategories of its bounded derived category, see \cite[Lemma~5.2]{Bruning07}.

\begin{lemma}
\label{L:BruningWideTheoremForMetrics}
Assume $R$ is a hereditary ring. Equip \(\mathcal{S}\) with an~additive good metric \(\mathcal{M}=\{B_n\}_{n\in\N}\). Let \(\mathcal{C}\subseteq\modf\dashmodule R\) and \(n\in\N\). If \(\mathcal{C}\subseteq B_{n+2}\) as objects of \(\mathcal{S}\), then \(\wide(\mathcal{C})\subseteq B_n\).
\end{lemma}

\begin{proof}
We show that the closure of \(\mathcal{C}\) under extensions, kernels and cokernels (in \(\ModR\)) lies in \(B_n\).
Let \(M,N\in\mathcal{C}\). Any extension of $M$ by $N$ already lies in \(B_{n+2}\subseteq B_n\) by the definition of a good metric.
Let \(f:M\rightarrow N\) be a homomorphism of $R$-modules.
Consider the distinguished triangle
\[
M\xrightarrow{f}N\rightarrow \cone f \rightarrow \Sigma M
\]
in \(\mathcal{S}\). Since \(R\) is hereditary, we get \(\cone f\simeq \Sigma \Ker f\oplus \Coker f\). As \(M\in B_{n+2}\), we have \(\Sigma M\in B_{n+1}\). And since \(B_{n+1}\) is closed under extensions and direct summands, and \(N\in B_{n+1}\), we obtain \(\Sigma\Ker f,\Coker f\in B_{n+1}\), allowing us to conclude the~proof with \(\Ker f,\Coker f\in B_{n}\).
\end{proof}

\begin{lemma}
\label{L:CohomologiesSplitOff}
Assume $R$ is a hereditary ring. Equip \(\mathcal{S}\) with an~additive good metric \(\mathcal{M}=\{B_n\}_{n\in\N}\).
Let \(\left(\mathcal{T}^{\leq0},\mathcal{T}^{>0}\right)\) be the standard $t$\=/structure on \(\mathcal{T}\).
Let \(M\in\ModR\) and \(T\in\mathcal{T}^{> 1}\). If \(M\oplus T\in\mathfrak{S}'(\mathcal{S})\), then \(T,M\in\mathfrak{S}'(\mathcal{S})\).
\end{lemma}

\begin{proof}
Assume \(M\oplus T\in\mathfrak{S}'(\mathcal{S})\).
By Corollary~\ref{C:UniformBoundedCohomologyCauchySequenceRestrictionGeneralised}, there exists a good Cauchy sequence \(\mathbf{E}=(E_n,e_n)_{n\in\N}\) such that for all \(n\in\N\) the object \(E_n\) has cohomology concentrated in non-negative degrees and \(\hoco \mathbf{E}\simeq M\oplus T\).
Suppose that \(M\oplus T\) is compactly supported at \(s\in\N\). Then by Corollary~\ref{C:CummingsGratzTrivialisationForSequences} we can even w.l.o.g.\ assume that \(\mathbf{E}\) is $s$-trivial.
Since
\[
\varinjlim H^1(\mathbf{E})\simeq H^1(\hoco \mathbf{E})\simeq H^1(M\oplus T)\simeq 0,
\] and all $R$-modules in the directed diagram \(H^1(\mathbf{E})\) are finitely generated, we obtain that for all \(n\in\N\) there exists \(m>n\) such that \(H^1(e_{n,m})=0\). This is why, we can (after restricting to a subsequence) w.l.o.g.\ assume that \(H^1(e_{n})=0\) for all \(n\in\N\).

As $R$ is of finite global dimension, it holds that \(\left(\mathcal{T}^{\leq0}\cap\mathcal{S},\mathcal{T}^{>0}\cap\mathcal{S}\right)\) is a~$t$\=/structure on \(\mathcal{S}\).
We consider two approximation sequences of \(\mathbf{E}\); the \(\mathcal{T}^{\leq 0}\)-approximation \(\mathbf{X}=(X_n,f_n)_{n\in\N}\) and the~\(\mathcal{T}^{> 0}\)\=/approximation \(\mathbf{Y}=(Y_n,g_n)_{n\in\N}\). For all \(n\in\N\) we have \(E_n\simeq X_n\oplus Y_n\) (because $R$ is hereditary), and the object \(X_n\) has cohomology concentrated in degree $0$.

Fix an index \(t\geq s\). Let \(N\in\N\) be such that for all \(n\geq N\) it holds that \(C_n:=\cone e_n\in B_{t+3}\). Pick \(n\geq N\).
The objects \(\Sigma^{-1}H^1(C_{n})\) and \(H^0(C_{n+1})\) are direct summands of \(C_n\) and \(C_{n+1}\), respectively, so they both lie in \(B_{t+3}\). Hence both the modules \(H^1(C_{n})\) and \(H^0(C_{n+1})\) lie in \(B_{t+2}\). 
Consider the distinguished triangles
\[
 E_n \xrightarrow{e_{n}} E_{n+1} \rightarrow C_n \rightarrow \Sigma E_n \text{\ \ \ and\ \ \ } E_{n+1} \xrightarrow{e_{n+1}} E_{n+2} \rightarrow C_{n+1} \rightarrow \Sigma E_{n+1}
.\]
The long exact sequences in cohomology 
\[
 H^1(E_n) \xrightarrow{0} H^1(E_{n+1}) \rightarrow H^1(C_n) \text{\ \ \ and\ \ \ }
 H^0(C_{n+1}) \rightarrow H^1(E_{n+1}) \xrightarrow{0} H^1(E_{n+2}) 
\]
express \(H^1(E_{n+1})\) as an image of an $R$-homomorphism \(H^0(C_{n+1})\rightarrow H^1(C_{n})\) of objects from \(B_{t+2}\). Lemma~\ref{L:BruningWideTheoremForMetrics} gives us \(H^1(E_{n+1})\in B_t\subseteq B_s\). As \(\mathbf{E}\) is assumed to be $s$-trivial, this forces \(H^1(E_{n+1})\simeq0\).

By~the~triangulated \(\ThreeThreeLemma\) (see~\cite[{Lemma~2.6}]{May01}), there exists a~diagram
\begin{center}
	\begin{tikzcd}
			X_{n+1} \arrow[r,""] \arrow[d,"f_{n+1}"] & E_{n+1} \arrow[d, "e_{n+1}"] \arrow[r] & Y_{n+1} \arrow[r, "0"] \arrow[d, "g_{n+1}"] & \Sigma X_{n+1}  \arrow[d]
			\\
			X_{n+2} \arrow[r,""] \arrow[d, ""] & E_{n+2}\arrow[r] \arrow[d,""] & Y_{n+2} \arrow[r, "0"] \arrow[d, ""] & \Sigma X_{n+2} \arrow[d]
            \\
            C^X_{n+1} \arrow[r, ""] \arrow[d, ""] & C_{n+1} \arrow[r] \arrow[d, ""] & C^Y_{n+1} \arrow[r] \arrow[d, ""] & \Sigma C^X_{n+1} \arrow[d, ""]
            \\
            \Sigma X_{n+1} \arrow[r,""] & \Sigma E_{n+1} \arrow[r] & \Sigma Y_{n+1} \arrow[r, "0"] & \Sigma^2 X_{n+1}
		\end{tikzcd}
\end{center}
where all rows and columns are distinguished triangles and all squares are commutative with the exception of the right bottom square which is anticommutative. Taking cohomologies and using \(H^0(\mathbf{X})=\mathbf{X}\) and \(H^0(\mathbf{Y})\simeq0\simeq H^{-1}(\mathbf{Y})\) gives us a~commutative diagram
\begin{center}
	\begin{tikzcd}
			0 \arrow[r] & H^{-1}(C^X_{n+1}) \arrow[r,""] \arrow[d,""] & H^{-1}(C_{n+1}) \arrow[d, ""] \arrow[r] & 0 \arrow[r] \arrow[d] & H^0(C^X_{n+1})  
			\\
			0 \arrow[r] & X_{n+1} \arrow[r,""] \arrow[d,"f_{n+1}"] & H^0(E_{n+1}) \arrow[d, "H^0(e_{n+1})"] \arrow[r] & 0 \arrow[r] \arrow[d] & 0  
			\\
			0 \arrow[r] & X_{n+2} \arrow[r,""] \arrow[d, ""] & H^0(E_{n+2}) \arrow[r] \arrow[d,""] & 0 \arrow[r, ""] \arrow[d, ""] & 0 
            \\
            0 \arrow[r] & H^0(C^X_{n+1}) \arrow[r, ""] \arrow[d, ""] & H^0(C_{n+1}) \arrow[r] \arrow[d, ""] & H^0(C^Y_{n+1}) \arrow[r] \arrow[d, ""] & 0 
            \\
            0 \arrow[r] & 0 \arrow[r,""] & H^1(E_{n+1}) \arrow[r] & H^1(Y_{n+1}) \arrow[r, ""] & 0
		\end{tikzcd}
\end{center}
of $R$-modules with exact rows and columns. Using exactness and \(H^1(E_{n+1})\simeq0\), we obtain \(H^{-1}(C^X_{n+1}) \simeq H^{-1}(C_{n+1})\in B_{t+1}\) and \(H^{0}(C^X_{n+1}) \simeq H^{0}(C_{n+1})\in B_{t}\).

In other words, we have \(C^X_{n+1}\simeq\Sigma H^{-1}(C^X_{n+1})\oplus H^{0}(C^X_{n+1})\in B_t\). Since $t$ and $n$ were arbitrary, we may conclude that \(\mathbf{X}\) is a Cauchy sequence. Because \[
\hoco \mathbf{X}\simeq \varinjlim H^0(\mathbf{E})\simeq H^0(\hoco \mathbf{E})\simeq H^0(M\oplus T)\simeq M
,\]
we get \(M\in\mathfrak{L}'(\mathcal{S})\). But $M$ is also compactly supported as a direct summand of an~element of the completion. Hence \(M\in\mathfrak{S}'(\mathcal{S})\).

Finally, consider the split distinguished triangle
\[
M \rightarrow M\oplus T \rightarrow T \rightarrow \Sigma M
\]
in \(\mathcal{T}\). Since \(M,M\oplus T\in\mathfrak{S}'(\mathcal{S})\) and \(\mathfrak{S}'(\mathcal{S})\) is a triangulated subcategory of~\(\mathcal{T}\), we obtain \(T\in\mathfrak{S}'(\mathcal{S})\).
\end{proof}

\begin{proposition}
\label{P:CompletionIsIdempotentComplete}
Assume $R$ is a hereditary ring, and equip \(\mathcal{S}\) with an~additive good metric. Then \(\mathfrak{S}'(\mathcal{S})\) is idempotent complete.
\end{proposition}

\begin{proof}
We start by observing that if two objects \(X,Y\) of the good extension \(\mathcal{T}\) satisfy \(X\oplus Y\in\mathfrak{S}'(\mathcal{S})\), then 
from the existence of the triangles
\[
X\oplus Y \xrightarrow{\begin{bsmallmatrix}\id_{X} & 0\\ 0 & 0\end{bsmallmatrix}} X\oplus Y \rightarrow Y\oplus \Sigma Y \rightarrow \Sigma X\oplus \Sigma Y
\]
and
\[
\Sigma Y\oplus \Sigma^2 Y \xrightarrow{\begin{bsmallmatrix}0 & 0\\ \id_{\Sigma Y} & 0\end{bsmallmatrix}} Y \oplus \Sigma Y \rightarrow Y \oplus \Sigma^3 Y \rightarrow \Sigma^2Y \oplus \Sigma^3 Y
\]
in \(\mathfrak{S}'(\mathcal{S})\) we can infer \(Y\oplus \Sigma Y^3\in\mathfrak{S}'(\mathcal{S})\).

Let \(C\in\mathfrak{S}'(\mathcal{S})\). Then $C$ is bounded in cohomology by Proposition~\ref{P:BoundedHereditaryCompletion}. Since $R$ is hereditary, this means \(C\simeq \coprod_{i=-l}^l \Sigma^{-i} H^i(C) \) for some \(l\in\N\). 
By the above observation, \(H^{-l}(C) \oplus \Sigma^{-3} H^{-l}(C)\in\mathfrak{S}'(\mathcal{S})\).
We get \(H^{-l}(C)\in\mathfrak{S}'(\mathcal{S})\) by Lemma~\ref{L:CohomologiesSplitOff}. An~induction argument on \(-l\leq i\leq l\) then proves that \(\mathfrak{S}'(\mathcal{S})\) is closed under taking cohomologies.

All we are left with to prove is that for modules \(M,N\in\ModR\) with \(M\in\mathfrak{S}'(\mathcal{S})\) and $N$ a direct summand of $M$, we have \(N\in\mathfrak{S}'(\mathcal{S})\). However, this follows from the same trick as before; the initial observation gives us \(N\oplus \Sigma^{-3}N\in\mathfrak{S}'(\mathcal{S})\), and Lemma~\ref{L:CohomologiesSplitOff} yields \(N\in\mathfrak{S}'(\mathcal{S})\).
We conclude that the completion \(\mathfrak{S}'(\mathcal{S})\) is idempotent complete.
\end{proof}

\begin{definition}
Let \(\mathcal{C}\) be a triangulated category. We say that \(\mathcal{C}\) is \textit{countably generated} if there exists an at most countable subset \(\mathcal{D}\subseteq\mathcal{T}\) such that \(\langle \mathcal{D}\rangle=\mathcal{T}\).
\end{definition}

We shall now show that (for $R$ hereditary) elements of a completion of \(\mathcal{S}\) also lie in a completion with respect to a specific countably generated constant metric. This is useful for reducing arguments about generic metrics to constant metrics.

\begin{lemma}
\label{L:LivingInCompletionGivenByThickSubcategory}
Assume $R$ is a hereditary ring. Equip \(\mathcal{S}\) with an~additive good metric \(\mathcal{M}=\{B_n\}_{n\in\N}\).
Let \(E\in\mathfrak{S}'_{\mathcal{M}}(\mathcal{S})\). Then for every \(t\in\N\) there exist \(s\geq t\) and~a~countably generated thick subcategory \(\mathcal{C}\) of~\(\mathcal{S}\) such that \(\mathcal{C}\subseteq\thick\big(H^0(B_s)\big)\) and \(E\in\mathfrak{S}'_{\mathcal{C}}(\mathcal{S})\).
\end{lemma}

\begin{proof}
By Proposition~\ref{P:BoundedHereditaryCompletion}, there exists \(l\in\N\) such that $E$ has cohomology concentrated in degrees \(-l,\ldots,l-1\). By Corollary~\ref{C:UniformBoundedCohomologyCauchySequenceRestrictionGeneralised}, we can find a Cauchy (with respect to \(\mathcal{M}\)) sequence  \(\mathbf{E}=(E_n,e_n)_{n\in\N}\) with \(E\simeq\hoco\mathbf{E}\) such that for all \(n\in\N\) the~entry \(E_n\) has cohomology concentrated in degrees \(-l,\ldots,l\).
Denote \(C_n:=\cone e_n\) for all \(n\in\N\).

Fix \(t\in\N\).
Let \(s\geq t\) be such that \(\mathbf{E}\) is compactly supported at \(s\). Let \(N\in\N\) be such that for all \(n\geq N\) we have \(C_n\in B_{s+2l+2}\). Define
\(
\mathcal{C}:=\thick\left(C_n: N \leq n\right)
.\)
Then \(\mathbf{E}\) is Cauchy with respect to \(\mathcal{C}\) and \(E\in\mathfrak{L}'_{\mathcal{C}}(\mathcal{S})\).

Since \(C_n=\coprod_{i=-l-1}^l\Sigma^{-i}H^i(C_n)\) for all \(n\geq N\), the category \(\mathcal{C}\) is also generated (as a thick subcategory) by
\[
\left\{H^i(C_n):-l-1\leq i\leq l,N\leq n\right\}\subseteq B_{s+l+1}\subseteq B_s
,\]
so \(\mathcal{C}\subseteq\thick\big(H^0(B_s)\big)\).

It remains to be shown that \(E\) is compactly supported at \(\mathcal{C}\). Pick indices \(n\geq N\), \(-l-1\leq i\leq l\) and \(j\in\Z\).
If \(-l-1\leq j\leq l\), then \(\Sigma^jH^i(C_n)\in\Sigma^jB_{s+l+1}\subseteq B_s\subseteq{}^{\perp}E\), so \(\Hom_{\mathcal{T}}\big(\Sigma^jH^i(C_n),E\big)=0\).
If \(j<-l-1\) or \(j>l\), then it once again holds that \(\Hom_{\mathcal{T}}\big(\Sigma^jH^i(C_n),E\big)=0\) because $E$ has cohomology concentrated in degrees \(-l,\ldots,l-1\) and $R$ is hereditary.

Therefore, all the shifts of the set \(
\{H^i(C_n):-l-1\leq i\leq l,N\leq n\}
\) of generators of \(\mathcal{C}\) lie in \({}^{\perp} E\). Because \({}^{\perp} E\) is closed under extensions,
we conclude that \(\mathbf{E}\) is compactly supported at \(\mathcal{C}\) and \(E\in\mathfrak{S}'_{\mathcal{C}}(\mathcal{S})\).
\end{proof}

In the next lemma, we employ Quillen's small object argument to build special Cauchy sequences with respect to constant metrics starting at a prescribed object by repeatedly taking suitable approximations.

\begin{lemma}
\label{L:SmallObjectArgumentCauchySequences}
Let \(\mathcal{C}\) be a countably generated triangulated subcategory of \(\mathcal{S}\). Then for every object \(X\in\mathcal{S}\) there exists a Cauchy sequence \(\mathbf{E}=(E_n,e_n)_{n\in\N}\) with respect to \(\mathcal{C}\) such that \(X\simeq E_1\), \(\hoco \mathbf{E}\in\mathfrak{S}'_{\mathcal{C}}(\mathcal{S})\), and \(\cone e_n\in\mathcal{C}\) for all \(n\in N\). 
\end{lemma}

\begin{proof}
Let \(\mathcal{C'}\) be an at most countable generating subset of \(\mathcal{C}\). Using a bijection \(\N\times\N\simeq\N\), we can find an enumeration \((C_n)_{n\in\N}\) of \(\mathcal{C'}\) such that for all \(C\in\mathcal{C}'\) and \(l\in\Z\) there exists an infinite subset \(I\subseteq\N\) satisfying \(\Sigma^lC\simeq C_i\) for each \(i\in I\).

Let \(X\in\mathcal{C}\). We construct the sequence \(\mathbf{E}=(E_n,e_n)_{n\in\N}\) by induction on \(n\). We start by defining \(E_1:=X\).
Assume that \(E_n\) is already constructed for some \(n\in\N\). Since \(\mathcal{S}\simeq\HomotopyCategory^b(\proj\dashmodule R)\), the $\kk$-module \(\Hom_{\mathcal{S}}(C_n,E_n)\) is generated by a finite subset \(J_n\) due to Lemma~\ref{L:PickingMaximalSummands}.
We then define \(e_n\) and \(E_{n+1}\) by the distinguished triangle
\[
\coprod_{j\in J_n}C_n \xrightarrow{\varphi_n} E_n \xrightarrow{e_{n}} E_{n+1} \rightarrow \coprod_{j\in J_n}\Sigma C_n
\]
where \(\varphi_n=[j]_{j\in J_n}\) is the natural coproduct map. This construction makes \(\varphi_n\) into a \(\coprodf(C_n)\)-precover. 
With this definition \(\cone e_n\in\mathcal{C}\) for all \(n\in\N\), so \(\mathbf{E}\) is indeed Cauchy.

Denote \(E:=\hoco\mathbf{E}\). We need to prove that $E$ is compactly supported with respect to \(\mathcal{C}\). Fix \(C\in\mathcal{C}'\) and \(l\in\Z\).
There are infinitely many indices \(m\in\N\) satisfying \(\Sigma^lC\simeq C_m\), and for any such an $m$ applying \(\Hom_{\mathcal{S}}(C_m,\blank)\) to the~distinguished triangle defining \(E_{m+1}\) provides us with an~exact sequence
\[
\Hom_{\mathcal{S}}\left(C_m,\coprod_{j\in J_m}C_m\right) \xrightarrow{\varphi_m\circ\blank} \Hom_{\mathcal{S}}(C_m,E_m) \xrightarrow{e_m\circ\blank} \Hom_{\mathcal{S}}(C_m,E_{m+1}) 
\]
where \(\varphi_m\circ\blank\) is clearly an epimorphism, so \(e_m\circ\blank=\Hom_{\mathcal{S}}(C_m,e_m)=0\) by the~exactness.
We infer \(0\simeq\varinjlim\Hom_{\mathcal{T}}\left(\Sigma^lC,\mathbf{E}\right)\simeq\Hom_{\mathcal{T}}(\Sigma^lC,E)\).

Consequently, as all the shifts of \(\mathcal{C}'\) lie in \({}^{\perp} E\), which is extension-closed, so does \(\mathcal{C}=\langle\mathcal{C}'\rangle\). Hence \(\mathbf{E}\) is compactly supported at \(\mathcal{C}\) and \(E\in\mathfrak{S}'_{\mathcal{C}}(\mathcal{S})\).
\end{proof}

\begin{remark}
Quillen's small object argument, where we construct precovers by taking coproducts over ``all possible morphisms'', has many applications in algebra and topology. In the context of triangulated categories, we can mention e.g.\ the~construction of compactly generated $t$-structures by Alonso, Jeremías, and Souto in \cite[Theorem~A.1]{AlonsoJeremiasJose03}.  
\end{remark}

For hereditary \(\kk\)-algebras, we can use Corollary~\ref{C:CompletionUniversalLocalisation} to express completions with respect to certain constant metrics as precisely bounded derived categories of corresponding universal localisations.

\begin{proposition}
\label{P:CountablyGeneratedCompletion}
Assume $R$ is a hereditary ring. Let \(\mathcal{C}\) be a countably generated thick subcategory of \(\mathcal{S}\).
Then \(\mathfrak{S}'_{\mathcal{C}}(\mathcal{S})\simeq\derived^b(\modf\dashmodule R_{\mathcal{C}})\).
\end{proposition}

\begin{proof}
We employ Lemma~\ref{L:SmallObjectArgumentCauchySequences} to the object \(X:=R\), and we obtain a Cauchy sequence \(\mathbf{E}=(E_n,e_n)_{n\in\N}\) in \(\mathcal{S}\) with respect to \(\mathcal{C}\) such that
\[
E:=\hoco \mathbf{E}\in\mathfrak{S}'_{\mathcal{C}}(\mathcal{S})
\text{,\ \ }
R\simeq E_1
\text{,\ \ and\ \ }
\cone e_n\in\mathcal{C} \text{ for all } n\in \N
.\]
Since \(E\in\derived(R_{\mathcal{C}})=\mathfrak{C}'_{\mathcal{C}}(\mathcal{S})\subseteq\mathcal{T}\) under the identification from Corollary~\ref{C:CompletionUniversalLocalisation}, it holds that \(E\simeq E\otimes^{\mathbf{L}}_{R}R_{\mathcal{C}}\).
After applying \(\blank\otimes^{\mathbf{L}}_{R}R_{\mathcal{C}}\) to the sequence \(\mathbf{E}\), we get a sequence \(\mathbf{E}\otimes^{\mathbf{L}}_{R}R_{\mathcal{C}}\) in the~category \(\derived(R_{\mathcal{C}})\). Because \(\cone e_n\in\mathcal{C}\subseteq\Ker\left(\blank\otimes^{\mathbf{L}}_{R}R_{\mathcal{C}}\right)\) for all \(n\in\N\), the sequence \(\mathbf{E}\otimes^{\mathbf{L}}_{R}R_{\mathcal{C}}\) consists of isomorphisms only, so
\[R_{\mathcal{C}}\simeq R\otimes^{\mathbf{L}}_{R}R_{\mathcal{C}}\simeq E_1\otimes^{\mathbf{L}}_{R}R_{\mathcal{C}}\simeq\hoco \mathbf{E}\otimes^{\mathbf{L}}_{R}R_{\mathcal{C}}\simeq E\otimes^{\mathbf{L}}_{R}R_{\mathcal{C}}\simeq E
.\]
As \(\mathfrak{S}'_{\mathcal{C}}(\mathcal{S})\) is idempotent complete by \ref{P:CompletionIsIdempotentComplete} and \(R_{\mathcal{C}}\in\mathfrak{S}'_{\mathcal{C}}(\mathcal{S})\), we get the desired equality \(\mathfrak{S}'_{\mathcal{C}}(\mathcal{S})=\derived^b(\modf\dashmodule R_{\mathcal{C}})\) by Proposition~\ref{P:HomologicalEpimorphismDeterminingCompletion}.
\end{proof}

\section{Lattice theory of metrics}
\label{sec:lattices}

This section is dedicated to developing the lattice theory of good metrics on a~triangulated category. We also present its application to the representation theory of algebras of finite representation type in Proposition~\ref{P:DynkinDecomposition} and hereditary rings in Proposition~\ref{P:DecompostionMetricsCompletion}.

Through this whole section \(\mathcal{T}\) will denote an essentially small triangulated category. Denote also \(\sim\) the equivalence of good metrics and \(\left[\mathcal{M}\right]\) the block of equivalence of a good metric \(\mathcal{M}\) on \(\mathcal{T}\) under \(\sim\).

\begin{notation}
Let be a triangulated category \(\mathcal{T}\) and \(\mathcal{A}\subseteq\mathcal{T}\) a class. We define \textit{extension closure} of \(\mathcal{A}\) as the smallest full subcategory of \(\mathcal{T}\) containing \(\mathcal{A}\) and $0$ which is closed under extensions. We denote it as \(\overline{\mathcal{A}}\).   
\end{notation}

\begin{proposition}
\label{P:LatticeExists}
Let \(\mathscr{M}\) be the class of all good metrics on \(\mathcal{T}\). Then \(\mathscr{M}/\sim\) forms a~lattice.

For good metrics \(\mathcal{A}=\{A_n\}_{n\in\N}\) and \(\mathcal{C}=\{C_n\}_{n\in\N}\) on \(\mathcal{T}\) the meet is given by \(\left[\mathcal{A}\right]\land\left[\mathcal{C}\right]:=\left[\{A_n\cap C_n\}_{n\in\N}\right]\) and the join is given by~\(\left[\mathcal{A}\right]\lor\left[\mathcal{C}\right]:=\left[\left\{\overline{A_n\cup C_n}\right\}_{n\in\N}\right]\).    
\end{proposition}

\begin{proof}
Let \(\mathcal{A}=\{A_n\}_{n\in\N},\) \(\mathcal{C}=\{C_n\}_{n\in\N}\in\mathscr{M}\). It is obvious that \(\{A_n\cap C_n\}_{n\in\N}\) is a~good metric.
We must check that \(\left[\{A_n\cap C_n\}_{n\in\N}\right]\) does not depend on the choice of the~representative elements for \(\left[\mathcal{A}\right]\) and \(\left[\mathcal{C}\right]\).

Suppose that \(\mathcal{A}\sim \mathcal{B}\) and \(\mathcal{C}\sim\mathcal{D}\) for some other good metrics \(\mathcal{B}=\{B_n\}_{n\in\N}\) and \(\mathcal{D}=\{D_n\}_{n\in\N}\) on \(\mathcal{T}\).
We want to prove \(\{A_n\cap C_n\}_{n\in\N}\sim\{B_n\cap D_n\}_{n\in\N}\).
Due to the~symmetry of~\(\cap\) and transitivity of \(\sim\), we may w.l.o.g.\ assume \(\mathcal{C}=\mathcal{D}\). Clearly, we have
\[
\{A_n\cap B_n \cap C_n\}_{n\in\N}\leq\{A_n\cap C_n\}_{n\in\N} \text{\ \ \ and\ \ \ } \{A_n\cap B_n \cap C_n\}_{n\in\N}\leq\{B_n\cap C_n\}_{n\in\N}
.\]
Since \(\mathcal{A}\sim\mathcal{B}\), for every \(n\in\N\) we can find \(m\geq n\) such that \(A_m\subseteq B_n\). Then
\[
A_m\subseteq A_n\cap B_n
\text{\ \ \ \ \ and\ \ \ \ \ }
A_m\cap C_m\subseteq A_n\cap B_n\cap C_m \subseteq A_n\cap B_n\cap C_n
.\]
Hence \(\{A_n\cap B_n \cap C_n\}_{n\in\N}\sim\{A_n\cap C_n\}_{n\in\N}\).
Similarly, with the role of \(\mathcal{A}\) and \(\mathcal{B}\) swapped, we obtain \(\{A_n\cap B_n \cap C_n\}_{n\in\N}\sim\{B_n\cap C_n\}_{n\in\N}\).

Consequently, 
\(\{A_n\cap C_n\}_{n\in\N}\sim\{B_n\cap C_n\}_{n\in\N}\), so \(\left[\{A_n\cap C_n\}_{n\in\N}\right]\) is choice-free.
Furthermore, it is clear that any metric finer than both \(\mathcal{A}\) and \(\mathcal{C}\) is finer than \(\{A_n\cap C_n\}_{n\in\N}\) as well.
Hence the meet \(\left[\mathcal{A}\right]\land\left[\mathcal{C}\right]\) exists and is equal to the desired equivalence class \(\left[\{A_n\cap C_n\}_{n\in\N}\right]\).

Now, we tackle the join operation. 
It is easy to see that \(\left\{\overline{A_n\cup C_n}\right\}_{n\in\N}\) is a~good metric.
We must now verify that \(\left[\left\{\overline{A_n\cup C_n}\right\}_{n\in\N}\right]\) does not depend on the~choice of~the~representative elements for \(\mathcal{A}\) and \(\mathcal{C}\).
Similarly as above, we may assume that \(\mathcal{A}\sim\mathcal{B}\) for some \(\mathcal{B}=\{B_n\}_{n\in\N}\in\mathscr{M}\). W.l.o.g.\ it is enough to show \(\left[\left\{\overline{A_n\cup C_n}\right\}_{n\in\N}\right]=\left[\left\{\overline{B_n\cup C_n}\right\}_{n\in\N}\right]\).
Fix \(n\in\N\). We find \(m\geq n\) such that \(A_m\subseteq B_n\). Then \(\overline{A_m\cup C_m}\subseteq \overline{A_m\cup C_n}\subseteq \overline{B_n\cup C_n}\).

Hence it holds that \(\left\{\overline{A_n\cup C_n}\right\}_{n\in\N}\leq\{\overline{B_n\cup C_n}\}_{n\in\N}\).
A dual argument then yields \(\left\{\overline{A_n\cup C_n}\right\}_{n\in\N}\sim\{\overline{B_n\cup C_n}\}_{n\in\N}\).

Because any metric coarser than both \(\mathcal{A}\) and \(\mathcal{C}\) is coarser than \(\left\{\overline{A_n\cup C_n}\right\}_{n\in\N}\) as well, we see that the join
\[
\left[\mathcal{A}\right]\lor\left[\mathcal{C}\right]:=\left[\left\{\overline{A_n\cup C_n}\right\}_{n\in\N}\right]
\]
exists and is well-defined. We conclude that \(\mathscr{M}/\sim\) has a structure of a lattice.
\end{proof}

\begin{definition}
By \(\M_{\mathcal{T}}\) we shall denote the~lattice given by Proposition~\ref{P:LatticeExists} applied to \(\mathcal{T}\) and call it \textit{a lattice of good metrics} on \(\mathcal{T}\).

By abuse of notation, we will (from now on) not distinguish between a good metric \(\mathcal{M}\) on \(\mathcal{T}\) and its equivalence class \([\mathcal{M}]\in\M_{\mathcal{T}}\) unless necessary. Also, for two good metrics \(\mathcal{M}\) and \(\mathcal{N}\) on \(\mathcal{T}\) we will write \(\mathcal{M}\lor\mathcal{N}\) and \(\mathcal{M}\land\mathcal{N}\) instead of \([\mathcal{M}]\lor[\mathcal{N}]\) and \([\mathcal{M}]\land[\mathcal{N}]\), respectively.
\end{definition}

\begin{lemma}
\label{L:CompactlySupportedWtrExtensionClosure}
Equip \(\mathcal{T}\) with two good metrics \(\mathcal{M}=\{B_n\}_{n\in\N}\) and \(\mathcal{N}=\{C_n\}_{n\in\N}\).
Then \(\mathfrak{C}_{\mathcal{M}\vee\mathcal{N}}(\mathcal{T})=\mathfrak{C}_{\mathcal{M}}(\mathcal{T})\cap\mathfrak{C}_{\mathcal{N}}(\mathcal{T})\).
\end{lemma}

\begin{proof}
This is clear, because for all \(F\in\Mod\dashmodule\mathcal{T}\) and a fixed \(n\in\N\) it holds that \(F\left(\overline{B_n\cup C_n}\right)=0\) if and only if \(F\left(B_n\right)=0\) and \(F\left(C_n\right)=0\).
\end{proof}

\begin{remark}
The lattice \(\M_{\mathcal{T}}\) is bounded with the constant metric \(\mathcal{T}\) being the~larg\-est element and the constant metric \(0\) being the smallest.

However, an infinite intersection is not an infinite meet in general because it is not compatible with equivalences of metrics.
More precisely, consider a metric \(\{B_n\}_{n\in\N}\) such that \(\bigcap_{n\in\N} B_n=0\) and \([\{B_{n}\}_{n\in\N}]\neq[0]\)  (e.g.~any aisle metric \(\left\{\mathcal{T}^{\leq n}\right\}_{n\in\N}\) for a~non\=/degenerate $t$-structure \(\left(\mathcal{T}^{\leq 0},\mathcal{T}^{>0}\right)\), see \cite[Example 12\,i)]{Neeman20},\cite[Definition~4.2]{CummingsGratz24}). 
Then all the metrics of~the~form \(\{B_{kn}\}_{n\in\N}\) for some \(k\in\N\) are equivalent.

Therefore, the meet satisfies \(\bigwedge_{k=1}^\infty[\{B_{kn}\}_{n\in\N}]=[\{B_n\}_{n\in\N}]\), but this block of equivalence is not equal to~\([\{\bigcap_{k=1}^\infty B_{kn}\}_{n\in\N}]=[\{0\}_{n\in\N}]\).
\end{remark}

We show that for a representation-finite hereditary algebra, additive good metrics can be decomposed into ``indecomposable metrics''. If $K$ is an algebraically closed field, every connected representation-finite hereditary algebra is Morita equivalent to the path algebra $KQ$ of a simply laced Dynkin quiver.
For any finite dimensional indecomposable \(KQ\)-module $M$ it then holds that \(\langle M\rangle=\coprodf\left\{\Sigma^iM:i\in\Z\right\}\) in \(\derived^b(\modf\dashmodule KQ)\).

\begin{proposition}
\label{P:DynkinDecomposition}
Let \(K\) be a field and $Q$ a simply laced Dynkin quiver.
Let \(\mathcal{N}:=\{B_n\}_{n\in\N}\) be an additive good metric on~the~category \(\derived^b(\modf\dashmodule KQ)\).
Then
\[
\mathcal{N}=\bigvee_{M\in\ind\dashmodule KQ}\big(\mathcal{N} \wedge \langle M\rangle\big)
\]
in~\(\M_{\derived^b(\modf\dashmodule KQ)}\).
\end{proposition}

\begin{proof} 
It is enough to prove that two specific representatives from each of~the~equivalence classes of good metrics are equivalent. We shall do so for the~metrics
\(\mathcal{N}\) and~\(\left\{\overline{\bigcup_{M\in\ind\dashmodule KQ} B_n\cap\langle M\rangle}\right\}_{n\in\N}\).

``$\leq$''
Let \(n\in\N\). For every indecomposable object from \(B_n\), which is necessiarily of the form \(\Sigma^iM\) for some \(M\in\ind\dashmodule KQ\) and \(i\in\Z\), we have \(\Sigma^iM\in B_n\cap\langle M\rangle\).

Since the~metric \(\mathcal{N}\) is additive, all objects from $B_n$ are extensions of indecomposables from \(B_n\). Therefore, we have
\[
B_n\subseteq \overline{\bigcup_{M\in\ind\dashmodule KQ} B_n\cap\langle M\rangle}
.\]
The proves \(
\mathcal{N}\leq\bigvee_{M\in\ind\dashmodule KQ}\big(\mathcal{N} \wedge \langle M\rangle\big)
.\)

``$\geq$''
Let \(n\in\N\).
For all \(M\in\ind\dashmodule KQ\), we have \(B_n\cap\langle M\rangle\subseteq B_n\).
Hence \(\overline{\bigcup_{M\in\ind\dashmodule KQ}B_n\cap\langle M\rangle}\subseteq B_n\) for any \(M\in\ind\dashmodule KQ\).

Consequently, we get \(\bigvee_{M\in\ind\dashmodule KQ}\big(\mathcal{N} \wedge \langle M\rangle\big)\leq\mathcal{N}\)
by the property of the join.
The claim now follows.
\end{proof}

\begin{example}
Let \(\{K_m\}_{m\in\N}\) be a countable family of fields. We shall work with~the~coproduct triangulated category \(\mathcal{S}:=\coprod_{m=1}^\infty\derived(K_m)\).
The indecomposables of~\(\mathcal{S}\) are precisely the fields of the form \(K_m\) for some \(m\in\N\) concentrated in~degree zero in the $m$-th component of the category, and all their shifts.

For every \(m\in\N\) we define a metric \(\{B^m_n\}_{n\in\N}\) by the formula
\[
B^m_n=\coprodf\left\{\Sigma^iK_m:i\geq n\right\}
\text{\ \ \ \ \ for all \(n\in\N\)}.
\]
We show that the infinite join \(\bigvee_{m=1}^\infty\{B^m_n\}_{n\in\N}\) does not exist using a diagonal argument.

Assume for a contradiction that there exists a metric \(\{C_n\}_{n\in\N}\) on \(\mathcal{T}\) satisfying \(\{C_n\}_{n\in\N}=\bigvee_{m=1}^\infty\{B^m_n\}_{n\in\N}\).
Since \(\{C_n\}_{n\in\N}\) is coarser than \(\{B^m_n\}_{n\in\N}\) for any choice of \(m\in\N\), there exists a strictly increasing function \(f_m:\N\rightarrow\N\) with \(B^m_{f_m(n)}\subseteq C_n\) for all \(n\in\N\).

The metric \(\left\{\overline{\bigcup_{m\in\N}B^m_{f_m(n)}}\right\}_{n\in\N}\leq\{C_n\}_{n\in\N}\) is coarser than \(\{B^m_n\}_{n\in\N}\) for all indices \(m\in\N\), so by the definition of a join we actually get an equivalence \(\left\{\overline{\bigcup_{m\in\N}B^m_{f_m(n)}}\right\}_{n\in\N}\sim\{C_n\}_{n\in\N}\).

Now consider the metric \(\left\{\overline{\bigcup_{m\in\N}B^m_{f_m(mn)}}\right\}_{n\in\N}\leq\{C_n\}_{n\in\N}\).
It is also coarser than \(\{B^m_n\}_{n\in\N}\) for all indices \(m\in\N\), so by the definition of a join, we also get an~equivalence \(\left\{\overline{\bigcup_{m\in\N}B^m_{f_m(n)}}\right\}_{n\in\N}\sim\left\{\overline{\bigcup_{m\in\N}B^m_{f_m(mn)}}\right\}_{n\in\N}\sim\{C_n\}_{n\in\N}\).

This means, in particular, that there exists \(n\in\N\) satisfying
\[
\overline{\bigcup_{m\in\N}B^m_{f_m(n)}}\subseteq\overline{\bigcup_{m\in\N}B^m_{f_m(m)}}
\]
but it is impossible because \(B^m_{f_m(n)}\not\subseteq B^m_{f_m(m)}\) for all \(m> n\).
We have reached a~contradiction.

Not only we have proven that \(\bigvee_{m=1}^\infty\{B^m_n\}_{n\in\N}\) does not exist, and the lattice \(\M_{\mathcal{S}}\) is not complete, but also that a~decomposition statement such as Proposition~\ref{P:DynkinDecomposition} cannot hold in general for triangulated categories with infinitely many indecomposables.
No metric on \(\mathcal{S}\) coarser than \(\{B^m_n\}_{n\in\N}\) for all \(m\in\N\) can be written as an~infinite join of \(\{B^m_n\}_{n\in\N}\)'s. 
\end{example}

\subsection{Decomposition of metrics}

\begin{definition}
Let \(\mathcal{S}\) be a thick subcategory of \(\mathcal{T}\). Let \((\mathcal{T}^{\leq0},\mathcal{T}^{>0})\) be a $t$-structure on \(\mathcal{T}\).
We define the~following good metrics (up to equivalence) on \(\mathcal{S}\):
\begin{enumerate}[label=(\roman*)]
\item 
\textit{the aisle metric} as the good metric \(\mathcal{S}_{-}:=\left\{\mathcal{T}^{\leq n}\cap\mathcal{S}\right\}_{n\in\N}\),
\item 
\textit{the coaisle metric} as the good metric \(\mathcal{S}_{+}:=\left\{\mathcal{T}^{> n}\cap\mathcal{S}\right\}_{n\in\N}\),
\item 
\textit{the t-structure metric} as the good metric \(\mathcal{S}_{\infty}:=\mathcal{S}_{-}\vee\mathcal{S}_{+}\) in \(\M_{\mathcal{S}}\).
\end{enumerate}

If the subcategory \(\mathcal{S}\) is not specified, we implicitly take \(\mathcal{S}=\mathcal{T}\).
\end{definition}

\begin{remark}
Similar metrics, namely metrics given by specific shifts of kernels of a fixed homological functor \(\mathcal{T}\rightarrow\mathcal{A}\) into an abelian category \(\mathcal{A}\), have been studied by Neeman in \cite[Example~12]{Neeman20}.
The notions of aisle and coaisle metrics were formally introduced by Cummings and Gratz in \cite[Definition~4.2]{CummingsGratz24}.
The notion of a~$t$\=/structure metric is new, while it may in some cases coincide with the metric from \cite[Example~3.4]{Neeman25A} called Krause's metric due its connection to Krause's construction of the sequential completion of a triangulated category from \cite{Krause20}.
\end{remark}

\begin{definition}
Let \(\mathcal{S}\) be a thick subcategory of \(\mathcal{T}\).
Let \((\mathcal{T}^{\leq0},\mathcal{T}^{>0})\) be a $t$-structure on~\(\mathcal{T}\). Let \(\mathcal{M}\) be a good metric on~\(\mathcal{S}\). 
We define the~following good metrics (up to equivalence) on \(\mathcal{S}\):
\begin{enumerate}[label=(\roman*)]
\item 
\textit{the aisle submetric}~of~\(\mathcal{M}\) as the good metric \(\mathcal{M}_{-}:=\mathcal{M}\wedge\mathcal{S}_{-}\) in~\(\M_{\mathcal{S}}\),
\item 
\textit{the coaisle submetric}~of~\(\mathcal{M}\) as the good metric \(\mathcal{M}_{+}:=\mathcal{M}\wedge\mathcal{S}_{+}\) in~\(\M_{\mathcal{S}}\),
\item 
\textit{the t-structure submetric}~of~\(\mathcal{M}\) as the good metric \(\mathcal{M}_{\infty}:=\mathcal{M}\wedge\mathcal{S}_{\infty}\) in~\(\M_{\mathcal{S}}\),
\end{enumerate}
\end{definition}

For bounded derived categories of hereditary algebras, the $t$-structure submetric of a good metric does not contribute much to the completion in the following sense:
\begin{proposition}
\label{P:DecompostionMetricsCompletion}
Let $R$ be a hereditary, finitely generated algebra over a commutative noetherian ring.
Equip \(\mathcal{S}:=\derived^b(\modf\dashmodule R)\) with the~standard $t$\=/structure \((\mathcal{S}^{\leq0},\mathcal{S}^{>0})\) and additive good metrics \(\mathcal{M}=\{B_n\}_{n\in\N}\) and \(\mathcal{N}=\{C_n\}_{n\in\N}\).
If \(\mathcal{M}=\mathcal{M}_{\infty}\vee\mathcal{N}\), then \(\mathfrak{S}'_{\mathcal{M}}(\mathcal{S})=\mathfrak{S}'_{\mathcal{N}}(\mathcal{S})\).
\end{proposition}

\begin{proof}
The (pre)completions and compactly supported elements are calculated inside the good extension \(\mathcal{S}\xhookrightarrow{}\derived(R)\) using Theorem~\ref{T:GoodExtensionComputationalTool}.

``$\supseteq$'' Let \(E\in\mathfrak{S}'_{\mathcal{N}}(\mathcal{S})\). Then \(E\in\mathfrak{L}'_{\mathcal{N}}(\mathcal{S})\subseteq\mathfrak{L}'_{\mathcal{M}}(\mathcal{S})\). By Proposition~\ref{P:BoundedHereditaryCompletion}, we know \(E\in\derived^b(\ModR)\), so \(E\in\mathfrak{C}'_{\mathcal{M}_{\infty}}(\mathcal{S})\).
Since also \(E\in\mathfrak{C}'_{\mathcal{N}}(\mathcal{S})\), by Lemma~\ref{L:CompactlySupportedWtrExtensionClosure} we have \(E\in\mathfrak{C}'_{\mathcal{M}_{\infty}\vee\mathcal{N}}(\mathcal{S})=\mathfrak{C}'_{\mathcal{M}}(\mathcal{S})\).
Thus \(E\in\mathfrak{S}'_{\mathcal{M}}(\mathcal{S})\).

``$\subseteq$''
 Let \(E\in\mathfrak{S}'_{\mathcal{M}}(\mathcal{S})\). Then \(E\in\mathfrak{C}'_{\mathcal{M}}(\mathcal{S})\subseteq\mathfrak{C}'_{\mathcal{N}}(\mathcal{S})\) by Lemma~\ref{L:CompactlySupportedWtrExtensionClosure}.
By Proposition~\ref{P:BoundedHereditaryCompletion}, there exists \(l\in\N\) such that $E$ has cohomology concentrated in degrees \(-l+1,\ldots,l-1\). By Corollary~\ref{C:UniformBoundedCohomologyCauchySequenceRestrictionGeneralised}, we can find a Cauchy sequence (with respect to~\(\mathcal{M}\)) \(\mathbf{E}=(E_n,e_n)_{n\in\N}\) with \(E\simeq\hoco\mathbf{E}\) such that for all \(n\in\N\) the~entry \(E_n\) has cohomology concentrated in degrees \(-l+1,\ldots,l\).

Let $m\in\N$.
We find \(N\in\N\) such that for all \(n\geq N\) we have
\[
\cone e_n\in\overline{\mathcal{S}^{\leq-l}\cup\mathcal{S}^{>l}\cup C_m}
.\]

After applying the double truncation functor \(\tr^{\leq l}\circ\tr^{>-l}\) we obtain that
\[
\cone e_n=\tr^{\leq l}\circ\tr^{>-l}(\cone e_n)\in\tr^{\leq l}\circ\tr^{>-l}\left(\overline{\mathcal{S}^{\leq-l}\cup\mathcal{S}^{>l}\cup C_m}\right)\subseteq C_m
.\]
This implies that \(\mathbf{E}\) stabilises at \(C_m\). As $m$ was arbitrary, \(\mathbf{E}\) is Cauchy with respect to \(\mathcal{N}\).
Hence \(E\in\mathfrak{L}'_{\mathcal{N}}(\mathcal{S})\) and consequently \(E\in\mathfrak{S}'_{\mathcal{N}}(\mathcal{S})\).
\end{proof}

The condition \(\mathcal{M}=\mathcal{M}_{\infty}\vee\mathcal{N}\) from Proposition~\ref{P:DecompostionMetricsCompletion} can be interpreted as some sort of local uniform convergence of \(\mathcal{M}\) towards \(\mathcal{N}\). Up to the exception of finitely many indices \(n\in\N\), objects with bounded cohomology cannot distinguish between the metrics \(\mathcal{M}=\{B_n\}_{n\in\N}\) and \(\mathcal{N}=\{C_n\}_{n\in\N}\).

The last lemma of this section describes a scenario where the conditions of Proposition~\ref{P:DecompostionMetricsCompletion} are not satisfied.

\begin{lemma}
\label{L:ObstructingObjects}
Let $R$ be a hereditary, finitely generated algebra over a commutative noetherian ring.
Equip \(\mathcal{S}:=\derived^b(\modf\dashmodule R)\) with the~standard $t$\=/structure \((\mathcal{S}^{\leq0},\mathcal{S}^{>0})\) and an additive good metric \(\mathcal{M}=\{B_n\}_{n\in\N}\).
If \(\mathcal{M}\not\sim\mathcal{M}_{\infty}\vee\mathcal{B}\), where \(\mathcal{B}:=\bigcap_{n\in\N} B_n\), then for every \(E\in\mathfrak{S}'_{\mathcal{M}}(\mathcal{S})\) there exist an additive good metric \(\mathcal{N}=\{C_n\}_{n\in\N}\sim\mathcal{M}\) on \(\mathcal{S}\) and a family of finitely generated $R$-modules \((X_n)_{n\in\N}\) such that for all \(n\in\N\):
\begin{itemize}
    \item \(\Hom_{\derived(R)}(\Sigma^jX_n,E)=0\) for all \(j\in\Z\),
    \item \(X_n\in C_n\),
    \item \(X_{n}\notin \thick\big(H^0(C_{n+1})\big)\).
\end{itemize}
\end{lemma}

\begin{proof}
The intersection \(\mathcal{B}=\bigcap_{n\in\N} B_n\) is, by \cite[Lemma~3.5]{Matousek26A}, a thick subcategory of \(\mathcal{S}\), so the expression \(\mathcal{M}_{\infty}\vee\mathcal{B}\) makes sense.
Since \(\mathcal{M}\gneq\mathcal{M}_{\infty}\vee\mathcal{B}\), there exists
\(l\in\N\) such that for all \(n\in\N\) we have
\(B_n\not\subseteq\overline{\mathcal{B}\cup\mathcal{S}^{\leq-l}\cup\mathcal{S}^{>l}}\).
The completion \(\mathfrak{S}'_{\mathcal{M}}(\mathcal{S})\) is calculated inside the good extension \(\mathcal{S}\xhookrightarrow{}\derived(R)\) using Theorem~\ref{T:GoodExtensionComputationalTool}.

Pick \(E\in\mathfrak{S}'_{\mathcal{M}}(\mathcal{S})\). Using Proposition~\ref{P:BoundedHereditaryCompletion}, we find \(k\in\N\) such that $E$ has cohomology concentrated in degrees \(-k,\ldots,k-1\). Suppose \(E\) is compactly supported at some \(s_1\in\N\).
Because \(B_{s_1+l+k}\not\subseteq\overline{\mathcal{B}\cup\mathcal{S}^{\leq-l}\cup\mathcal{S}^{>l}}\), we can find \(X_1\in\modf\dashmodule R\) and \(-l\leq i\leq l\) such that \(\Sigma^iX_1\in B_{s_1+l+k}\setminus\overline{\mathcal{B}\cup\mathcal{S}^{\leq -l}\cup\mathcal{S}^{>l}}\).

Then \(X_1\in\Sigma^{-i}B_{s_1+l+k}\subseteq B_{s_1+k}\), and \(\Sigma^jX_1\in B_{s_1}\) for all \(-k\leq j\leq k\). Due to \(E\in B_{s_1}^{\perp}\) and the bounds on the~cohomology of $E$, we get \(\Hom_{\derived(R)}(\Sigma^jX_1,E)=0\) for all \(j\in\Z\).
Because \(X_1\notin \mathcal{B}\), there exists \(s_2>s_1+2\) such that \(X_1\notin B_{s_2-2}\). Then \(X_1\notin\thick\big(H^0(B_{s_2})\big)\) since otherwise we would have \(X_1\in\wide\big(H^0(B_{s_2})\big)\subseteq B_{s_2-2}\) due to Lemma~\ref{L:BruningWideTheoremForMetrics}.

We continue the construction of \(s_n\) and \(X_n\) for general \(n\in\N\) inductively by the~same process as above.
We finish the proof by defining \(\mathcal{N}:=\{B_{s_n}\}_{n\in\N}\) and noticing \(\mathcal{N}\sim\mathcal{M}\).
\end{proof}

\section{Completions for commutative noetherian rings}
\label{sec:CommutativeNoetherian}

This section contains one of the main results of this paper (Theorem~\ref{T:DedekindCompletions}); that is an explicit description of all completions of the bounded derived category of a~hereditary commutative noetherian ring with respect to additive good metrics.

\subsection{General completion lemmas}

Before we start the actual computation of the completions in our special case, we need the following two facts about general completions of triangulated categories.

\begin{lemma}
\label{L:CompactlySupportedElementOfTheOriginalCategory}
Let \(\mathcal{S}\) be a triangulated category equipped with a good metric and \(X\in\mathcal{S}\).
If \(\Yoneda(X)\in\mathfrak{C}(\mathcal{S})\), then \(\Yoneda(X)\in\mathfrak{S}(\mathcal{S})\).
\end{lemma}

\begin{proof}
There is an equality \(\Yoneda(X)=\moco\mathbf{E}\) for~the~constant Cauchy sequence \(\mathbf{E}=\left(X,\id_{X}\right)_{n\in\N}\) in \(\mathcal{S}\), so \(\Yoneda(X)\in\mathfrak{L}(\mathcal{S})\).
Since we already know that \(\Yoneda(X)\) is compactly supported, we get \(\Yoneda(X)\in\mathfrak{S}(\mathcal{S})\).
\end{proof}

The above argument directly translates into the setting of good extensions, giving us that \(\mathfrak{C}'(\mathcal{S})\cap\mathcal{S}\subseteq\mathfrak{S}'(\mathcal{S})\).

\begin{lemma}
\label{L:SequenceOfComplactlySupportedElements}
Equip a triangulated category \(\mathcal{S}\) with a good metric \(\mathcal{M}=\{B_n\}_{n\in\N}\). Let \(\mathbf{E}=(E_n,e_n)_{n\in\N}\) be a Cauchy sequence compactly supported at \(s\in\N\), and \(E:=\moco \mathbf{E}\).  Let \(\mathbf{F}=(F_n,f_n)_{n\in\N}\) be any (non-necessiarily Cauchy) sequence satisfying \(F_n\in B_s^{\perp}\) for all \(n\in\N\).
If \(E\simeq\moco \mathbf{F}\), then \(E\) is a direct summand in \(\Mod\dashmodule\mathcal{S}\) of an object from \(\Yoneda(\mathcal{S})\).
\end{lemma}

\begin{proof}
By Lemma~\ref{L:BasicFactorisationProperty}, there exists \(N\in\N\) such that
for all \(n\geq N\) and all
\[
X\in\{\Yoneda(F_m):m\in\N\}\cup\{E\}
\]
every map \(\Yoneda(E_n)\rightarrow X\) in \(\Mod\dashmodule\mathcal{T}\) factors uniquely through the~colimit injection \(\varphi_n:\Yoneda(E_n)\xrightarrow{} E\).

By Yoneda's lemma and \(E\simeq\moco \mathbf{F}\), we get \(\varphi_N:\Yoneda(E_N)\rightarrow E\) factors
as \(\Yoneda(E_N)\xrightarrow{g}\Yoneda(F_M)\xrightarrow{\psi_M}E\) for some \(M\in\N\) where \(\psi_M\) is the colimit injection.
By the choice of $N$, the map $g$ factors as \(\Yoneda(E_N)\xrightarrow{\varphi_N}E\xrightarrow{h}\Yoneda(F_M)\).
But then the~equality
\(\id_E\varphi_N=\psi_Mg=\psi_Mh\varphi_N\) gives us two factorisations of the same map \(\Yoneda(E_N)\rightarrow E\) through \(\varphi_M\). Since this factorisation is unique, we obtain \(\id_E=\psi_Mh\), thus exhibiting $E$ as a direct summand of \(\Yoneda(F_M)\).
\end{proof}

Once again, the above lemma translates into the setting of good extensions. In particular, if \(\mathcal{T}\) is a triangulated category with coproducts and \(\mathcal{T}^c\xhookrightarrow{}\mathcal{T}\) the~corresponding good extension, then for any sequence \(\mathbf{E}\) satisfying the conditions of Lemma~\ref{L:SequenceOfComplactlySupportedElements}, it necessiarily holds that \(\hoco\mathbf{E}\in\mathcal{T}^c\subseteq\mathcal{T}\) because \(\mathcal{T}^c\) is idempotent complete.

\subsection{The computation}
As every hereditary commutative noetherian ring is a~direct product of Dedekind domains, we can reduce all the computations for completions to Dedekind domains only - in a general case we can proceed componentwise. 
Therefore, through this whole section we shall work with a~De\-de\-kind domain $D$ and \(\mathcal{S}:=\derived^b(\modf\dashmodule D)\simeq\perfect(D)\) its bounded derived category. We denote \(\mathcal{T}:=\derived(D)\) and note that the inclusion \(\mathcal{S}\xhookrightarrow{}\mathcal{T}\) is a~good extension, so all (pre)completions will be calculated inside \(\mathcal{T}\).

We shall denote \(\mathfrak{T}\subseteq\mathcal{S}\) the thick subcategory generated by all finitely generated torsion $D$-modules.

\begin{lemma}
\label{L:UncountablyGeneratedCompletion}
Let \(\mathcal{D}\subsetneq\mathcal{C}\) be thick subcategories of \(\mathcal{S}\).
Then:
\begin{enumerate}[label=(\roman*)]
\item It holds that \(\mathfrak{S}'_{\mathcal{C}}(\mathcal{S})\cap\mathfrak{S}'_{\mathcal{D}}(\mathcal{S})\subseteq\mathcal{C}^{\perp}\cap\mathfrak{T}\).
\item If \(\mathcal{C}\) is not countably generated, then \(\mathfrak{S}'_{\mathcal{C}}(\mathcal{S})=\mathcal{C}^{\perp}\cap\mathfrak{T}\).
\end{enumerate}
\end{lemma}

\begin{proof}
``Part i)''
Let \(E\in\mathfrak{S}'_{\mathcal{C}}(\mathcal{S})\cap\mathfrak{S}'_{\mathcal{D}}(\mathcal{S})\).
We find a sequence \(\mathbf{E}=(E_n,e_n)_{n\in\N}\) with \(E\simeq\hoco\mathbf{E}\) which is Cauchy with respect to \(\mathcal{D}\), and therefore with respect to \(\mathcal{C}\) as well. W.l.o.g.\ suppose \(\cone e_n\in\mathcal{D}\) for all \(n\in\N\).
By Corollary~\ref{C:CummingsGratzTrivialisationForSequences}, we can also w.l.o.g.\ assume that \(\mathbf{E}\) is \(\mathcal{C}\)-trivial.
Theorem~\ref{T:TelescopeConjectureCommutativeNoetherian} yields \(\Supp(\mathcal{D})\subsetneq\Supp(\mathcal{C})\).
We pick a prime ideal \(\mathfrak{p}\in\Supp(\mathcal{C})\setminus\Supp(\mathcal{D})\). Then \(\mathcal{D}\otimes^{\mathbf{L}}_Dk(\mathfrak{p})=0\).

Assume for contradiction that there exists \(m\in\N\) and \(i\in\Z\) such that \(H^i(E_m)\) has a projective direct summand.
Since $D$ is connected, Lemma~\ref{L:ProjectiveSupportIsLocallyConstant} gives us \(\Supp(E_m)=\Spec(D)\).
As \(\blank\otimes^{\mathbf{L}}_Dk(\mathfrak{p})\) annihilates \(\cone e_n\) for all \(n\in\N\), the sequence \(\mathbf{E}\otimes^{\mathbf{L}}_Dk(\mathfrak{p})\) in \(\derived(D)\) is constant, so
\[
E\otimes^{\mathbf{L}}_Dk(\mathfrak{p})\simeq\hoco\mathbf{E}\otimes^{\mathbf{L}}_Dk(\mathfrak{p})\simeq E_m\otimes^{\mathbf{L}}_Dk(\mathfrak{p})\neq0
.\]
But that is impossible because \(\Supp\big(\mathfrak{C}'_{\mathcal{C}}(\mathcal{S})\big)=\Supp(D_{\mathcal{C}})=\Spec(D)\setminus\Supp(\mathcal{C})\) by~\cite[Lemma~3.5]{Neeman92A}.

We now know that no entry of \(\mathbf{E}\) contains a shift of a projective module as a~direct summand. Because \(\mathbf{E}\) is also \(\mathcal{C}\)-trivial, we get by Theorem~\ref{T:DedekindFinitelyGeneratedModules} that every entry of \(\mathbf{E}\) is supported at \(\Supp(D_{\mathcal{C}})\).
In other words, every entry of \(\mathbf{E}\) is
compactly supported at \(\mathcal{C}\). Taking \(\mathbf{F}=\mathbf{E}\) in Lemma~\ref{L:SequenceOfComplactlySupportedElements} then gives us that \(E\in\mathcal{S}\). We already know that \(\Supp(E)\neq\Spec(D)\), so $E$ cannot contain a shift of a projective module as a direct summand. Thus we may conclude \(E\in\mathcal{C}^{\perp}\cap\mathfrak{T}\).

``Part ii)'' Assume furthermore that \(\mathcal{C}\) is not countably generated. Let us have \(E\in\mathfrak{S}'_{\mathcal{C}}(\mathcal{S})\).
By Lemma~\ref{L:LivingInCompletionGivenByThickSubcategory}, there exists a countably generated thick subcategory \(\mathcal{D}\subseteq\mathcal{C}\) such that \(E\in\mathfrak{S}'_{\mathcal{D}}(\mathcal{S})\). Since \(\mathcal{C}\) is not countably generated, we have \(\mathcal{D}\subsetneq\mathcal{C}\). Hence part i) applies, so \(E\in\mathcal{C}^{\perp}\cap\mathfrak{T}\).
This proves \(\mathfrak{S}'_{\mathcal{C}}(\mathcal{S})\subseteq\mathcal{C}^{\perp}\cap\mathfrak{T}\).
However, the~other inclusion \(\mathcal{C}^{\perp}\cap\mathfrak{T}\subseteq\mathfrak{S}'_{\mathcal{C}}(\mathcal{S})\) holds by Lemma~\ref{L:CompactlySupportedElementOfTheOriginalCategory}.
\end{proof}

We now formulate and prove the main result. Recall that for a good metric \(\mathcal{M}=\{B_n\}_{n\in\N}\), the intersection \(\mathcal{B}:=\bigcap_{n\in\N} B_n\) is a triangulated subcategory by \cite[Lemma~3.5]{Matousek26A}. Furthermore, if \(\mathcal{M}\) is additive, then \(\mathcal{B}\) is thick.

\begin{theorem}
\label{T:DedekindCompletions}
Equip \(\mathcal{S}\) with the standard $t$-structure \((\mathcal{S}^{\leq0},\mathcal{S}^{>0})\) and an additive good metric \(\mathcal{M}=\{B_n\}_{n\in\N}\). Consider the~thick subcategory \(\mathcal{B}:=\bigcap_{n\in\N} B_n\) of~\(\mathcal{S}\). 
Then:
\begin{enumerate}[label=(\roman*)]
\item If \(\mathcal{M}=\mathcal{M}_{\infty}\vee\mathcal{B}\), and if \(\mathcal{B}\) is countably generated, then
\[
\mathfrak{S}'_{\mathcal{M}}(\mathcal{S})=\mathfrak{S}'_{\mathcal{B}}(\mathcal{S})=\derived^b(\modf\dashmodule D_{\mathcal{B}})\subseteq\derived(D)
.\]
\item Otherwise, we have \(\mathfrak{S}'_{\mathcal{M}}(\mathcal{S})=\mathcal{B}^{\perp}\cap\mathfrak{T}\subseteq\mathcal{S}\).
\end{enumerate}
\end{theorem}

\begin{proof}
If the condition of ``case i)'' happens, the result follows from Proposition~\ref{P:DecompostionMetricsCompletion} and Proposition~\ref{P:CountablyGeneratedCompletion}.
For the rest of the proof, we then assume that the condition of~``case i)'' is not satisfied. There are two possible explanations.

The first possibility is that \(\mathcal{M}=\mathcal{M}_{\infty}\vee\mathcal{B}\) and \(\mathcal{B}\) is not countably generated.
Then we obtain
\[
\mathfrak{S}'_{\mathcal{M}}(\mathcal{S})=\mathfrak{S}'_{\mathcal{B}}(\mathcal{S})=\mathcal{B}^\perp\cap\mathfrak{T}
\]
by Proposition~\ref{P:DecompostionMetricsCompletion} and Lemma~\ref{L:UncountablyGeneratedCompletion}.

All we are left with is to prove \(\mathfrak{S}'_{\mathcal{M}}(\mathcal{S})=\mathcal{B}^\perp\cap\mathfrak{T}\) when \(\mathcal{M}\gneq\mathcal{M}_{\infty}\vee\mathcal{B}\).

``$\supseteq$''
Let \(T\in\mathcal{B}^\perp\cap\mathfrak{T}\).
We can w.l.o.g.\ assume that \(0\neq T\) is an indecomposable torsion module supported at a single prime ideal \(\mathfrak{p}\in\Spec(D)\) concentrated in degree zero, otherwise we just build $T$ as an~extension of suitable shifts of indecomposable summands of its cohomologies.

By~Lemma~\ref{L:CompactlySupportedElementOfTheOriginalCategory}, it is enough to prove that \(T\in\mathfrak{C}'_{\mathcal{M}}(\mathcal{S})\).
Using \(T\notin \mathcal{B}\), we find \(n\in\N\) such that \(T,\Sigma^{-1}T\notin B_n\).
If there existed \(X\in H^0(B_{n+2})\) such that \(\Hom_{\derived(D)}(X,T)\neq0\), then \(\mathfrak{p}\in\Supp(X)\), so \(T\in\wide(X)\subseteq B_n\) by Lemma~\ref{L:BruningWideTheoremForMetrics}, which is impossible.
Analogously, it is possible to show that there does not exist \(X\in H^1(B_{n+2})\) such that \(\Hom_{\derived(D)}(\Sigma^{-1}X,T)\neq0\).

This means \(H^0(B_{n+1}),\Sigma^{-1}H^1(B_{n+1})\subseteq{}^{\perp}T\).
Since there are no non-zero morphisms from \(\mathcal{S}^{\leq-1}\) and \(\mathcal{S}^{>1}\) to $T$, we see that $T$ is compactly supported at \(n+2\). This finishes the proof of this inclusion.

``$\subseteq$''
Let \(E\in\mathfrak{S}'_{\mathcal{M}}(\mathcal{S})\). By Lemma~\ref{L:LivingInCompletionGivenByThickSubcategory}, there exist \(s\in\N\) and a thick subcategory \(\mathcal{E}\subseteq\thick\big(H^0(B_s)\big)\) such that \(E\in B_s^{\perp}\) and \(E\in\mathfrak{S}'_{\mathcal{E}}(\mathcal{S})\).

Since \(\mathcal{M}\gneq\mathcal{M}_{\infty}\vee\mathcal{B}\), we can by Lemma~\ref{L:ObstructingObjects} w.l.o.g.\ (up to replacing \(\mathcal{M}\) by an~equivalent metric) assume, that there exists a $D$-module \(X\in B_s\) such that
\[
X\notin\thick\big(H^0(B_{s+1})\big) \text{\ \ \ and\ \ \ }\Hom_{\derived(R)}(\Sigma^jX,E)=0 \text{\ \ \ for all\ \ \ } j\in\Z
.\]
We use Lemma~\ref{L:LivingInCompletionGivenByThickSubcategory} again to find a thick subcategory \(\mathcal{D}\subseteq\thick\big(H^0(B_{s+1})\big)\) with \(E\in\mathfrak{S}'_{\mathcal{D}}(\mathcal{S})\).
Then for \(\mathcal{C}:=\thick\big(\mathcal{D}\cup\mathcal{E}\cup\{X\}\big)\) we have \(\mathcal{D}\subsetneq\mathcal{C}\) and \(E\in\mathfrak{S}'_{\mathcal{C}}(\mathcal{S})\).
Lemma~\ref{L:UncountablyGeneratedCompletion} then gives us \(E\in\mathcal{C}^{\perp}\cap\mathfrak{T}\).
We finish the proof by observing that \(E\in\mathcal{B}^{\perp}\) because $E$ is compactly supported.
\end{proof}

\begin{example}
If \(D\) is a Dedekind domain of an uncountable spectrum, then \(\mathfrak{S}'_{\mathfrak{T}}(\mathcal{S})=0\).
\end{example}

\begin{example}
\label{E:CompletionsIntegers}
Set \(D=\Z\) and \(\mathcal{M}=\langle\Z/2\Z\rangle\) as a constant metric on~\(\mathcal{S}\). Then \(\mathfrak{S}'_{\mathcal{M}}(\mathcal{S})=\derived^b\left(\modf\dashmodule \Z[\frac{1}{2}]\right)\).
One of the possible Cauchy sequences expressing the~localisation ring \(\Z[\frac{1}{2}]\) as an element of \(\mathfrak{L}'_{\mathcal{M}}(\mathcal{S})\) is \(\Z\xrightarrow{\cdot2}\Z\xrightarrow{\cdot2}\cdots\). Similarly, for the~constant metric \(\mathfrak{T}\), we get \(\mathfrak{S}'_{\mathfrak{T}}(\mathcal{S})=\derived^b\left(\modf\dashmodule \Q\right)\).
But for a~non\=/constant metric \(\mathcal{N}=\left\{\left\langle\Z/p\Z:p\text{ prime}, p\geq n\right\rangle\right\}_{n\in\N}\) we get  
\(\mathfrak{S}'_{\mathcal{N}}(\mathcal{S})=\mathfrak{T}\).
\end{example}

\begin{example}
Let \(D\) be a Dedekind domain with \(\Pic(D)\) non-torsion, i.e.\ there exists a projective maximal ideal \(I\leq R\) such that \(I^n\) is not a principal ideal for all \(n\in\N\). There is an ideal \(J\leq R\) such that \(IJ=(a)\) is principal (generated by \(a\in D\)). Set \(\mathcal{M}=\langle D/I\rangle\). Then \(\mathfrak{S}'_{\mathcal{M}}(\mathcal{S})=\derived^b\left(\modf\dashmodule R_{\mathcal{M}}\right)\) and the universal localisation \(R_{\mathcal{M}}\) can be expressed as a homotopy colimit of the Cauchy sequence \(D\xrightarrow{\cdot a} J\xrightarrow{\cdot a}J^2\xrightarrow{\cdot a}\cdots\).
Unlike Example~\ref{E:CompletionsIntegers}, here, the universal localisation \(R_{\mathcal{M}}\) is not a classical localisation at a multiplicative subset of $D$ due to \cite[Chapitre~IV, Proposition~4.6]{Lazard69}.     
\end{example}

\section{Completions for tame algebras}
\label{sec:TameAlgebras}
In this final section, we present the second main result of this paper (Theorem~\ref{T:TameCompletions}) about completions of the bounded derived category of a~hereditary tame finite dimensional algebra.
The section follows the same structure as Section~\ref{sec:CommutativeNoetherian} but with the statements and proofs adjusted to the setting of finite dimensional algebras.

Through this whole section we work with an Euclidean quiver $Q$ and its path algebra \(KQ\) over an algebraically closed feld $K$.
Any connected hereditary tame finite dimensional algebra over $K$ is Morita equivalent to a path algebra over an~Euclidean quiver, so we may reduce our argumentation from the tame algebra to the~case of~$KQ$. 
We once again denote \(\mathcal{S}:=\derived^b(\modf\dashmodule KQ)\simeq\perfect(KQ)\), \(\mathcal{T}:=\derived(KQ)\) and note that the inclusion \(\mathcal{S}\xhookrightarrow{}\mathcal{T}\) is a good extension, so all (pre)completions will be calculated inside \(\mathcal{T}\). 

We recall that \(\mathcal{P}\), \(\mathcal{R}\), and \(\mathcal{Q}\) stand for the preprojective, regular, and preinjective components of \(\ind\dashmodule KQ\), respectively, and that \(\mathfrak{R}=\thick(\mathcal{R})\subseteq\mathcal{S}\).

\begin{lemma}
\label{L:UncountablyGeneratedCompletionTame}
Let \(\mathcal{D}\subsetneq\mathcal{C}\) be thick subcategories of \(\mathcal{S}\) generated by regular modules. Let \(S\in\ind\dashmodule KQ\) be a simple regular module from a homogenous tube.
Then:
\begin{enumerate}[label=(\roman*)]
\item If \(S\in\mathcal{C}\setminus\mathcal{D}\), then it holds that \(\mathfrak{S}'_{\mathcal{C}}(\mathcal{S})\cap\mathfrak{S}'_{\mathcal{D}}(\mathcal{S})\subseteq\mathcal{C}^{\perp}\cap\mathfrak{R}\).
\item If \(\mathcal{C}\) is not countably generated, then \(\mathfrak{S}'_{\mathcal{C}}(\mathcal{S})=\mathcal{C}^{\perp}\cap\mathfrak{R}\). 
\end{enumerate}
\end{lemma}

\begin{proof}
``Part i)''
Let \(E\in\mathfrak{S}'_{\mathcal{C}}(\mathcal{S})\cap\mathfrak{S}'_{\mathcal{D}}(\mathcal{S})\).
We find a sequence \(\mathbf{E}=(E_n,e_n)_{n\in\N}\) with \(E\simeq\hoco\mathbf{E}\) which is Cauchy with respect to \(\mathcal{D}\), and therefore with respect to \(\mathcal{C}\) as well. W.l.o.g.\ suppose \(\cone e_n\in\mathcal{D}\) for all \(n\in\N\).
By Corollary~\ref{C:CummingsGratzTrivialisationForSequences}, we can also w.l.o.g.\ assume that \(\mathbf{E}\) is \(\mathcal{C}\)-trivial.

By the mutual orthogonality of tubes, we have \(\mathcal{D}\subseteq S^{\perp}\).
This means that \(\Hom_{\mathcal{S}}(S,\blank)\) annihilates \(\cone e_n\) for all \(n\in\N\), so the sequence \(\Hom_{\mathcal{S}}(S,\mathbf{E})\) is constant. Hence \(\Hom_{\mathcal{S}}(S,E_n)\simeq\varinjlim\Hom_{\mathcal{S}}(S,\mathbf{E})\simeq\Hom_{\mathcal{S}}(S,E)\simeq0\) for all \(n\in\N\). As $S$ is a simple regular from a homogenous tube, we have \(\Hom_{KQ}(S,Q)\neq0\) and \(\Ext^1_{KQ}(S,P)\neq0\) for all \(Q\in\mathcal{Q}\) and \(P\in\mathcal{P}\) by Lemma~\ref{L:SincereModule}. This implies that \(H^0(E)\) does not contain a preinjective direct summand, and \(H^{-1}(E)\) does not contain a preprojective direct summand.

By repeating the same argument for arbitrary shifts of $S$, we obtain that for all \(n\in\N\) the cohomology of \(E_n\) consists of regular modules only. We then even have \(H^i(E_n)\in\mathcal{C}^\perp\) for all \(n\in\N\) and \(i\in\Z\) by the \(\mathcal{C}\)-triviality of \(\mathbf{E}\) and the orthogonality of tubes.
Taking \(\mathbf{F}=\mathbf{E}\) in Lemma~\ref{L:SequenceOfComplactlySupportedElements} then gives us that \(E\in\mathcal{S}\). In the same manner as before, we can consequently deduce that $E$ cannot contain a shift of neither a preprojective, nor preinjective module as a direct summand. Thus we may conclude \(E\in\mathcal{C}^{\perp}\cap\mathfrak{R}\).

``Part ii)'' Assume furthermore that \(\mathcal{C}\) is not countably generated.
Then there exists some uncountable indexing set \(I\subseteq\ProjectiveLine{K}\) such that
\(
\mathcal{C}=\coprod_{i\in I}\coprod_{j\in\Z}\Sigma^j\tube'_i,
\)
where \(0\neq\tube'_i\) is a wide subcategory of the tube \(\tube_i\) and all but finitely many tubes \(\tube_i\) are homogenous.
Let us have \(E\in\mathfrak{S}'_{\mathcal{D}}(\mathcal{S})\).
By Lemma~\ref{L:LivingInCompletionGivenByThickSubcategory}, there exists a countably generated thick subcategory \(\mathcal{D}\subseteq\mathcal{C}\) such that \(E\in\mathfrak{S}'_{\mathcal{D}}(\mathcal{S})\).
Here \(\mathcal{D}\) is necessiarily contained in a thick subcategory generated by countably many tubes.  
Since \(I\) is uncountable, we can find \(i\in I\) such that \(\tube_i=\tube_i'\) is homogenous and \(\tube_i\cap\,\mathcal{D}=0\). Hence part i) applies, so \(E\in\mathcal{C}^{\perp}\cap\mathfrak{R}\).
This proves \(\mathfrak{S}'_{\mathcal{C}}(\mathcal{S})\subseteq\mathcal{C}^{\perp}\cap\mathfrak{R}\).
However, the~other inclusion \(\mathcal{C}^{\perp}\cap\mathfrak{R}\subseteq\mathfrak{S}'_{\mathcal{C}}(\mathcal{S})\) holds by Lemma~\ref{L:CompactlySupportedElementOfTheOriginalCategory}.
\end{proof}

\begin{theorem}
\label{T:TameCompletions}
Equip \(\mathcal{S}\) with the standard $t$-structure \((\mathcal{S}^{\leq0},\mathcal{S}^{>0})\) and an additive good metric \(\mathcal{M}=\{B_n\}_{n\in\N}\). Consider the~thick subcategory \(\mathcal{B}:=\bigcap_{n\in\N} B_n\) of~\(\mathcal{S}\). 
Then:
\begin{enumerate}[label=(\roman*)]
\item If \(\mathcal{M}=\mathcal{M}_{\infty}\vee\mathcal{B}\), and if \(\mathcal{B}\) is countably generated, then
\[
\mathfrak{S}'_{\mathcal{M}}(\mathcal{S})=\mathfrak{S}'_{\mathcal{B}}(\mathcal{S})=\derived^b(\modf\dashmodule KQ_{\mathcal{B}})\subseteq\derived(KQ)
.\]
\item Otherwise, we have \(\mathfrak{S}'_{\mathcal{M}}(\mathcal{S})=\mathcal{B}^{\perp}\cap\mathfrak{R}\subseteq\mathcal{S}\).
\end{enumerate}
\end{theorem}

\begin{proof}
If the condition of ``case i)'' happens, the result follows from Proposition~\ref{P:DecompostionMetricsCompletion} and Proposition~\ref{P:CountablyGeneratedCompletion}.
For the rest of the proof, we then assume that the condition of~``case i)'' is not satisfied. There are two possible explanations.

The first possibility is that \(\mathcal{M}=\mathcal{M}_{\infty}\vee\mathcal{B}\) and \(\mathcal{B}\) is not countably generated. As \(\mathcal{B}\) then cannot be generated by an exceptional sequence, it holds that \(\mathcal{B}\subseteq\mathfrak{R}\) by Theorem~\ref{T:RegularOrExceptional}.
Then we obtain
\[
\mathfrak{S}'_{\mathcal{M}}(\mathcal{S})=\mathfrak{S}'_{\mathcal{B}}(\mathcal{S})=\mathcal{B}^\perp\cap\mathcal{S}
\]
by Proposition~\ref{P:DecompostionMetricsCompletion} and Lemma~\ref{L:UncountablyGeneratedCompletionTame}.

All we are left with is to prove \(\mathfrak{S}'_{\mathcal{M}}(\mathcal{S})=\mathcal{B}^\perp\cap\mathfrak{R}\) when \(\mathcal{M}\gneq\mathcal{M}_{\infty}\vee\mathcal{B}\).

``$\supseteq$''
Let \(T\in\mathcal{B}^\perp\cap\mathfrak{R}\).
We can w.l.o.g.\ assume that \(0\neq T\) is an indecomposable regular module from a single tube concentrated in degree zero, otherwise we just build $T$ as an~extension of suitable shifts of indecomposable summands of its cohomologies.

By~Lemma~\ref{L:CompactlySupportedElementOfTheOriginalCategory}, it is enough to prove that \(T\in\mathfrak{C}'_{\mathcal{M}}(\mathcal{S})\).
For all \(n\in\N\) we define \(\mathfrak{H}_n:=\big\{X\in\ind\dashmodule KQ:X\in B_{n+2},\Hom_{KQ}(X,T)\neq0\big\}\).
By Lemma~\ref{L:BruningWideTheoremForMetrics}, for every \(n\in\N\) it holds that \(\wide(\mathfrak{H}_n)\subseteq B_n\).
Then we have an infinite decreasing chain \(\thick(\mathfrak{H}_1)\supseteq\thick(\mathfrak{H}_2)\supseteq\cdots\) of thick subcategories of \(\mathcal{S}\). If this chain did not stabilise, then a repeated application of Lemma~\ref{L:InfiniteChainDiffersAtHomogenousTubes} would give us two different homogenous tubes \(\mathbf{u},\mathbf{v}\not\subseteq{}^{\perp}T\), which is impossible for a regular module $T$.
Hence the chain of thick subcategories above stabilises, meaning that \(\mathfrak{H}_n=\mathfrak{H}_m\) for some \(N\in\N\) and all \(m,n\geq N\), and \(\mathfrak{H}_N\subseteq\mathcal{B}\). If \(\mathfrak{H}_N\) was not zero, then it could not hold that \(T\in\mathcal{B}^{\perp}\). Therefore, we obtain \(H^0(B_{N})\subseteq{}^{\perp}T\).
Analogously, it is possible to show that \(H^1(B_{M})\subseteq{}^{\perp}T\) for some \(M\in\N\).

Since there are no non-zero morphisms from \(\mathcal{S}^{\leq-1}\) and \(\mathcal{S}^{>1}\) to $T$, we see that $T$ is compactly supported at \(\max\{N,M\}\). This finishes the proof of this inclusion.

``$\subseteq$''
We find \(M\in\N\) such that for any strictly decreasing chain of thick subcategories of \(\mathcal{S}\) some homogenous regular tube must appear at one stage and disappear at a~later one as (more precisely) stated in Lemma~\ref{L:InfiniteChainDiffersAtHomogenousTubes}.

Let \(E\in\mathfrak{S}'_{\mathcal{M}}(\mathcal{S})\).
Since \(\mathcal{M}\gneq\mathcal{M}_{\infty}\vee\mathcal{B}\), we can by Lemma~\ref{L:ObstructingObjects} w.l.o.g.\ (up to replacing \(\mathcal{M}\) by an~equivalent metric) assume, that there exist $KQ$-modules \(X_1,\ldots,X_M\) such that for all \(1\leq i\leq M\) we have
\(
X_{i}\in B_{i}
,\)
\(
X\notin\thick\big(H^0(B_{i+1})\big)
\)
and
\(
\Hom_{\derived(R)}(\Sigma^jX_i,E)=0 \text{ for every } j\in\Z 
.\)

By Lemma~\ref{L:LivingInCompletionGivenByThickSubcategory}, there exist thick subcategories \(\mathcal{E}_1,\ldots,\mathcal{E}_M\) of \(\mathcal{S}\) satisfying for each \(1\leq i\leq M\) that \(\mathcal{E}_i\subseteq\thick\big(H^0(B_i)\big)\) and \(E\in\mathfrak{S}'_{\mathcal{E}_i}(\mathcal{S})\).
Then for each \(1\leq i\leq M\) we define
\(\mathcal{C}_i:=\thick\left(\bigcup_{j=i}^M\left(\mathcal{E}_j\cup\{X_j\}\right)\right)\), ensuring that \(E\in\mathfrak{S}'_{\mathcal{C}_i}(\mathcal{S})\).
Now we have a strictly decreasing chain \(\mathcal{C}_1\supsetneq\mathcal{C}_2\supsetneq\cdots\supsetneq\mathcal{C}_M\) of thick subcategories of \(\mathcal{S}\), so Lemma~\ref{L:InfiniteChainDiffersAtHomogenousTubes} provides us with indices \(1\leq i<j\leq M\) and a homogenous regular tube \(\tube\subseteq\modf\dashmodule KQ\) such that \(\mathcal{C}_i,\mathcal{C}_j\subseteq\mathfrak{R}\), \(\tube\subseteq\mathcal{C}_i\) and \(\tube\cap\,\mathcal{C}_j=0\).

The conditions of Lemma~\ref{L:UncountablyGeneratedCompletionTame}, part~i), are now satisfied. Thus \(E\in\mathcal{C}_i^{\perp}\cap\mathfrak{R}\). 
We finish the proof by observing that \(E\in\mathcal{B}^{\perp}\) because $E$ is compactly supported.
\end{proof}

We see that if we want the completion \(\mathfrak{S}'_{\mathcal{M}}(\mathcal{S})\) to contain new elements which are not from the original category \(\mathcal{S}\), then the good metric \(\mathcal{M}\) must be of the form \(\mathcal{M}=\mathcal{M}_{\infty}\vee\mathcal{B}\) with \(\mathcal{B}\) countably generated. However, this condition is not sufficient. 

\begin{corollary}
\label{C:FiniteDimensionalUniversalLocalisation}
Equip \(\mathcal{S}\) with the standard $t$-structure \((\mathcal{S}^{\leq0},\mathcal{S}^{>0})\) and an additive good metric \(\mathcal{M}=\{B_n\}_{n\in\N}\). Assume that the~thick subcategory \(\mathcal{B}:=\bigcap_{n\in\N} B_n\) of~\(\mathcal{S}\) is generated by an exceptional sequence in \(\modf\dashmodule KQ\). If \(\mathcal{M}=\mathcal{M}_{\infty}\vee\mathcal{B}\), then \(\mathfrak{S}'_{\mathcal{M}}(\mathcal{S})=\mathcal{B}^{\perp}\cap\mathcal{S}\).     
\end{corollary}

\begin{proof}
Theorem~\ref{T:TameCompletions}~i) gives us \(\mathfrak{S}'_{\mathcal{M}}(\mathcal{S})=\mathfrak{S}'_{\mathcal{B}}(\mathcal{S})=\derived^b(\modf\dashmodule KQ_{\mathcal{B}})\). The category \(H^0(\mathcal{B})\) is generated by an exceptional sequence as a wide subcategory of \(\modf\dashmodule KQ\), so it is equivalent to a~module category. Hence \(H^0(\mathcal{B})=\wide(P)\) for a~rigid object \(P\in\modf\dashmodule KQ\). By \cite[Proposition~4.2]{KrauseStovicek10}, this forces the universal localisation \(KQ_{\mathcal{B}}\) to be finite dimensional over $K$. It follows that \(\derived^b(\modf\dashmodule KQ_{\mathcal{B}})\subseteq\derived^b(\modf\dashmodule KQ)\), and consequently \(\mathfrak{S}'_{\mathcal{M}}(\mathcal{S})=\mathcal{B}^{\perp}\cap\mathcal{S}\).
\end{proof}

In light of the above corollary, we may conclude the discussion of completions of tame algebras by constating that new objects appear in the completion only for specific metrics determined by countable collections of regular modules. If \(\mathcal{B}\) is generated by preprojective or preinjective modules, it is generated by an exceptional sequence, and the completion \(\mathfrak{S}'_{\mathcal{M}}(\mathcal{S})\) is then a thick subcategory of \(\mathcal{S}\).

For a hereditary finite dimensional algebra $A$ of finite representation type, it is known by \cite[Corollary~3.9]{Matousek26A} that any completion of
\(\derived^b(\modf\dashmodule A)\) is just a thick subcategory \(\derived^b(\modf\dashmodule A)\). One can obtain an alternative proof of this fact by mimicking the argument of Corollary~\ref{C:FiniteDimensionalUniversalLocalisation} - all indecomposable modules of \(A\) are both preprojective and preinjective, so no new object can appear in the completion. 

\begin{example}
\label{E:Kronecker}
Let $Q$ be the Kronecker quiver with a fixed orientation \(\xymatrix{ 2\ar@<3.0pt>[r]\ar@<-3.0pt>[r]& 1}\).
Every regular tube is homogenous, and the simple regulars are of the form 
\[
\xymatrix{R_{(1:k)}=K\ar@<3.0pt>[r]^-{1}
\ar@<-3.0pt>[r]_-{k}&K}
\text{\ \ and\ \ }
\xymatrix{R_{(0:1)}=K\ar@<3.0pt>[r]^-{0}
\ar@<-3.0pt>[r]_-{1}&K}
\text{\ \ for every\ \ }
k\in K
.\]
Fix \(\lambda\in\ProjectiveLine{K}\), and denote \(\mathcal{C}:=\langle R_{\lambda}\rangle\subseteq\mathcal{S}\).
The preprojective modules are precisely
\[
\xymatrix{P_n=K^n\ar@<3.0pt>[rr]^-{\begin{bsmallmatrix}\id_{n\times n}\\0\end{bsmallmatrix}}
\ar@<-3.0pt>[rr]_-{\begin{bsmallmatrix}0\\ \id_{n\times n}\end{bsmallmatrix}}&&K^{n+1}}
\]
for \(n\in\N_0\).
We can arrange the whole preprojective component into a Cauchy sequence \(\mathbf{E}=(P_{n-1},e_n)_{n\in\N}\) with respect to \(\mathcal{C}\) where \(\cone\left(P_{n-1}\xrightarrow{e_n}P_n\right)\simeq R_{\lambda}\) for all \(n\in\N\).
It holds that \(E:=\hoco \mathbf{E}\in\derived(KQ)\) is compactly supported at \(\mathcal{C}\). Under the identification \(\mathfrak{C}'_{\mathcal{C}}(\mathcal{S})\simeq\derived(KQ_{\mathcal{C}})\) within \(\derived(KQ)\), the universal localisation \(KQ_{\mathcal{C}}\) corresponds to \(\End_{KQ}(E\oplus E)\), which is Morita equivalent to \(\End_{KQ}(E)\). We then have
\(
\mathfrak{S}'_{\mathcal{C}}(\mathcal{S})\simeq\derived^b\left(\modf\dashmodule KQ_{\mathcal{C}}\right)\simeq\left\langle E,R_{\mu}\right\rangle_{\lambda\neq\mu\in\ProjectiveLine{K}}.
\) 
\end{example}

\begin{remark}
For an arbitrary Euclidean quiver $Q$, we can build Cauchy sequences of preprojectives with respect to thick subcategories generated by simple regulars in a fashion similar to Example~\ref{E:Kronecker}. For every \(P\in\mathcal{P}\) and \(R\) a simple regular, there exists a \(P'\in\mathcal{P}\) and a short exact sequence \(0\rightarrow P\rightarrow P'\rightarrow R\rightarrow0\) yielding a corresponding distinguished triangle in \(\mathcal{S}\). The construction of such Cauchy sequences has appeared under the name of \textit{special chains} in \cite[Lemma~7.3]{HugelSarochTrlifaj18}.
\end{remark}

\bibliographystyle{plain}
\bibliography{literatura}

\end{document}